\documentclass{amsart}
\usepackage[english]{babel}

\usepackage{amssymb} 
\usepackage{mathtools}
\usepackage{amsthm} 
\usepackage{amsmath}
\usepackage[all]{xy} 
\usepackage{xcolor} 
\usepackage{tikz}
\usepackage{tikz-cd}
\usepackage{calc}
\usetikzlibrary{calc}
\usepackage{enumitem}

\usepackage{thmtools}

\usepackage{hyphenat} 
\usepackage{csquotes}
\usepackage{hyperref}
\hypersetup{
    final=true,
    plainpages=false,
    pdfstartview=FitV,
    pdftoolbar=true,
    pdfmenubar=true,
    bookmarksopen=true,
    bookmarksnumbered=true,
    breaklinks=true,
    linktocpage,
    colorlinks=true,
    linkcolor=teal,
    urlcolor=teal,
    citecolor=teal,
    anchorcolor=green
  }

\usepackage[
style = alphabetic, 
sorting=nyt, 
url=true,
maxnames=10,
minnames=10, 
maxalphanames=10]{biblatex}
\DeclareFieldFormat[article]{volume}{\mkbibbold{#1}\addcolon \addspace}
\DeclareFieldFormat[article]{number}{\mkbibparens{#1}}
\renewbibmacro*{in:}{}

\numberwithin{equation}{section}

\renewcommand{\to}{\longrightarrow}

\newcommand{\R}{{\mathbb{R}}}

\theoremstyle{definition}
\newtheorem{definition}{}[section]

\theoremstyle{plain}
\newtheorem{proposition}[definition]{}

\newtheorem{theorem}[definition]{}

\newtheorem{lemma}[definition]{}

\theoremstyle{definition}
\newtheorem{example}[definition]{}

\theoremstyle{remark}
\newtheorem{remark}[definition]{}

\renewcommand{\eqref}[1]{\hyperref[#1]{(\protect\NoHyper\ref{#1}\protect\endNoHyper)}}

\DeclareCiteCommand{\cite}
{}
{\bibhyperref{[{\protect\NoHyper\usebibmacro{cite}}\protect\endNoHyper\usebibmacro{postnote}]}}
{   
    \multicitedelim%
}
{}

\newcommand{\Ric}{\mathrm{Ric}}
\newcommand{\Spin}{\mathrm{Spin}}

\newcommand{\Id}{\mathrm{Id}}

\newcommand{\grad}{\mathrm{grad}}

\newcommand{\U}{\mathrm{U}}

\title{
On generalized imaginary \texorpdfstring{Spin$^c$}{Spin^c}-Killing spinors}
\date{}

\author[J. L. Carmona Jim\'enez]{José Luis {\small CARMONA JIMÉNEZ}}
\address{JLCJ: 
Institute of Mathematics “Simion Stoilow” of the Romanian Academy, 21 Calea Grivitei, 010702 Bucharest, Romania}
\email{jcarmona@imar.ro}

\thanks{
\noindent JLCJ has been supported by the PNRR-III-C9-2023-I8 grant CF 149/31.07.2023 {\em Conformal Aspects of Geometry and Dynamics} and by the project PID2021-126124NB-I00 (Spain).}

\usepackage[a4paper]{geometry}
\begin{document}

\begin{abstract}
A non-trivial spinor field $\psi$ is called a generalized imaginary $\mathrm{Spin}^c$-Killing spinor if $\nabla^{g,A} _X \psi = i\mu X \cdot \psi$ for all vector fields $X$, where $\mu$ is a real function that is not identically zero and $\nabla^{g,A}$ is the $\mathrm{Spin}^c$ Levi-Civita connection with $\U(1)$-connection $A$.
Associated with $\psi$ is a vector field $V$, the Dirac current, defined by $g(V,X) = i \langle X\cdot \psi, \psi \rangle$.
We prove that if $V$ vanishes somewhere and $\dim M \geq 3$, the manifold is locally isometric to real hyperbolic space.
When $V$ never vanishes and $\operatorname{dim} M \geq 3$, we obtain a global geometric description of all $\mathrm{Spin}^c$-Riemannian manifolds carrying such spinors, under the assumption that either the normalized Dirac current $\xi = \frac{V}{|V|}$ is complete or the leaves of $\mathcal{D} = \ker(\xi^\flat)$ are complete. Finally, we reinterpret the case of type~I generalized imaginary $\mathrm{Spin}^c$-Killing spinors in terms of parallel spinors for a suitable connection with vectorial torsion.
\end{abstract}

\maketitle




\section{Introduction}

Spinor fields on Riemannian manifolds provide a natural link between differential geometry, topology, and mathematical physics.
Their existence imposes strong geometric restrictions and, in many cases, characterizes distinguished classes of manifolds.
Classical examples include parallel spinors, which force a reduction of the Riemannian holonomy group and lead to special geometries such as Calabi--Yau, hyper-Kähler, $G_2$-manifolds, and $\Spin(7)$-manifolds; see~\cite{W1989,MS2000}.
More generally, Killing spinors encode rigid geometric information and play a central role in the study of Einstein metrics and special geometric structures.
This raises the question of how much of this rigidity persists in the $\Spin^c$ setting, particularly without assuming completeness.

In this article we study generalized imaginary $\Spin^c$-Killing spinors on $\Spin^c$-Riemannian manifolds.
More precisely, if $(M,g)$ is a $\Spin^c$-Riemannian manifold, a non-trivial spinor field $\psi$ is called a \emph{generalized imaginary $\Spin^c$-Killing spinor} if there exists a real-valued function $\mu$ not identically zero such that
\[
    \nabla^{g,A}_X\psi=i\mu\,X\cdot\psi
\]
for all vector fields $X$, where $\nabla^{g,A}$ denotes the $\Spin^c$ Levi-Civita connection with $\U(1)$-connection $A$.
The function $i\mu$ is called the Killing function of $\psi$.
Associated with $\psi$ is a distinguished vector field $V$, its \emph{Dirac current}, defined by
\[
    g(V,X)= i\langle X\cdot\psi,\psi\rangle,\qquad \forall X\in \mathfrak{X}(M).
\]
A key feature is that $q_\psi = |\psi|^4 - |V|^2$ is a non-negative constant, see~\cite{GN2015}.
Hence, the alternatives $q_\psi=0$ and $q_\psi>0$ determine two types of generalized imaginary $\Spin^c$-Killing spinors, \emph{type~I} if $|V|=|\psi|^2$ and \emph{type~II} if $|V|< |\psi|^2$.

The complete spin case is by now well understood.
Imaginary Killing spinors with constant Killing number were classified by Baum~\cite{B1989,B1989*}, while the case of non-constant Killing function was described by Rademacher~\cite{R1991}.
In the $\Spin^c$ setting, generalized imaginary $\Spin^c$-Killing spinors on complete manifolds were studied by Große and Nakad~\cite{GN2015}; in particular, in dimension $n \geq 3$, the type~II case reduces to the statement that $(M,g)$ is globally isometric to the real hyperbolic space, while the type~I case admits a warped-product description in terms of manifolds carrying parallel $\Spin^c$-spinors on the universal cover.
More recently, Lockman~\cite{Lockman2025} analyzed the type~I case under the assumption that the normalized Dirac current is complete.
These results already point to a sharp contrast between the two types in the complete Riemannian $\Spin^c$ setting.

By contrast, the problem is much less understood without ambient completeness, even in the spin setting.
In type~II, the complete theory essentially reduces to the real hyperbolic space, whereas examples without ambient completeness show that this rigidity disappears once ambient completeness is dropped; see~\cite[pp.~161--162]{BHMMM15}.
This indicates that completeness of the ambient manifold is not the most natural global hypothesis for capturing the geometry dictated by the spinor. 
The relevant geometric objects are instead the normalized Dirac current $\xi=\frac{V}{|V|}$, defined where $V\neq0$, and the orthogonal distribution $\mathcal{D}=\ker(\xi^\flat)$. 
Accordingly, we work under the assumption that either $\xi$ is complete or the leaves of $\mathcal{D}$ are complete.

Our first goal is to describe the global geometry of $\Spin^c$-Riemannian manifolds carrying generalized imaginary $\Spin^c$-Killing spinors under these natural hypotheses.
We prove that, in dimension $n\geq 3$, the existence of such a spinor imposes strong global restrictions on $(M,g)$.
First, whenever the Dirac current vanishes somewhere, no completeness assumption is needed: the manifold is locally isometric to real hyperbolic space.
Second, whenever the Dirac current is nowhere vanishing and either $\xi$ is complete or the leaves of $\mathcal{D}$ are complete, we obtain a global geometric description depending on the type of the spinor.
When $\psi$ is of type~I, we prove (in~\ref{thm:typeI-global}) a warped-product or mapping-torus description, extending the corresponding type~I results in~\cite{GN2015,Lockman2025} to the setting without ambient completeness. 
When $\psi$ is of type~II, we show in~\ref{thm:typeII-global} that the Killing function is constant.
We also prove that either $(M,g)$ is locally isometric to real hyperbolic space or is globally isometric to a warped product
\[
    \bigl((0,t_m)\times F,\ dt^2+\sinh^2(2\mu t)\,h\bigr).
\]
In the latter case, $(F,h)$ is either a $6$-dimensional strictly nearly K\"ahler manifold or a $(4k+1)$-dimensional Einstein--Sasakian manifold, with $k\in\mathbb N$, $k\geq1$. 

Our second goal is to reinterpret type~I generalized imaginary $\Spin^c$-Killing spinors using metric connections with vectorial torsion. 
Given a unit vector field $\xi$ and a function $\alpha$, we define the $(\alpha,\xi)$-connections by
\[
    \nabla:=\nabla^g-S, \qquad S_XY:= \alpha \left( g(Y,\xi)\,X-g(X,Y)\xi\right).
\]
Under natural assumptions, we show that type~I generalized imaginary $\Spin^c$-Killing spinors are equivalent to $\nabla$-parallel $\Spin^c$-spinors.
This provides a vectorial-connection formulation of the type~I case, in the same spirit as the interpretation of real Killing spinors through distinguished metric connections with skew-symmetric torsion, see~\cite{FI2002,A2006}.

The paper is organized as follows.
Section~\ref{sec:preliminaries} collects the preliminaries used throughout the paper: spinorial conventions, metric connections with vectorial torsion, and the conformal viewpoint on closed Weyl structures.
Section~\ref{sec:imaginary-killing} contains the basic definitions, motivation, and examples of generalized imaginary $\Spin^c$-Killing spinors.
In Section~\ref{sec:main} we prove the global geometric description of $\Spin^c$-Riemannian manifolds carrying such spinors, distinguishing the cases where the Dirac current vanishes and where it does not.
Finally, in Section~\ref{sec:type-I} we establish the correspondence between generalized imaginary $\Spin^c$-Killing spinors of type~I and $(\alpha,\xi)$-connections admitting parallel $\Spin^c$-spinors.

\subsection*{Acknowledgments} I would like to thank Ilka Agricola for suggesting this topic and for very useful discussions. I am also grateful to Andrei Moroianu for his comments and a careful reading of a preliminary version of the manuscript.


\section{Preliminaries}\label{sec:preliminaries}
In this section, we fix the spinorial conventions and recall the basic facts about vectorial connections and conformal changes used later.
We refer to~\cite{F2000,BHMMM15} for the definitions of spin and $\Spin^c$-structures and for the construction of the associated spinor bundles.

From now on, $(M,g)$ will denote a connected oriented $\Spin^c$-Riemannian manifold of dimension $n$, with auxiliary principal $\U(1)$-bundle $P_{\U(1)}M\to M$, determinant line bundle
\[
    \mathbb{L}:=P_{\U(1)}M\times_{\U(1)} \mathbb C,
\]
and associated spinor bundle $\Sigma_g^c M$.
We fix a connection $A$ on the auxiliary principal $\U(1)$-bundle $P_{\U(1)}M\to M$.
By the same symbol $A$, we also denote the induced Hermitian connection on $\mathbb{L}$.
The $\Spin^c$-connection induced by the Levi-Civita connection $\nabla^g$ and the auxiliary connection $A$ is denoted by $\nabla^{g,A}$.
We denote by $\Omega^A\in \Omega^2(M,\mathbb R)$ the curvature $2$-form of $A$.
Locally, if $\omega\in \Omega^1(U,\mathbb R)$ is the connection $1$-form of $A$ with respect to a local section of $P_{\U(1)}M$, then $\Omega^A|_U=d\omega$.
In such a local trivialization, for every $X\in \mathfrak X(M)$ and $\psi\in \Gamma(\Sigma_g^c M)$, the induced $\Spin^c$-connection can be written as
\[
    \nabla^{g,A}_X \psi=\nabla^g_X \psi+\frac{i}{2}\omega(X) \psi,
\]
where $\nabla^g$ denotes the spinorial Levi-Civita connection. 
We denote the musical isomorphisms determined by the metric $g$ by
\[
    (\cdot)^\sharp \colon T^*M \longrightarrow TM, \qquad (\cdot)^\flat \colon TM \longrightarrow T^*M.
\]
Recall that these maps are $\nabla^g$-parallel and satisfy $(\cdot)^\sharp\circ(\cdot)^\flat = \Id_{TM}$ and $(\cdot)^\flat\circ(\cdot)^\sharp = \Id_{T^*M}$.

The usual spin case is recovered by considering the $\Spin^c$ structure induced by a spin structure on $M$.
In this case the determinant line bundle is canonically trivial, $\mathbb{L}\cong M\times \mathbb C$, and the associated $\Spin^c$-spinor bundle is canonically identified with the usual spinor bundle, $\Sigma_g^c M\cong \Sigma M$. Moreover, for the canonical flat Hermitian connection $A_0$ on $\mathbb{L}$, the induced connection $\nabla^{g,A_0}$ coincides with the usual spinorial Levi-Civita connection.
More generally, if $A$ is flat, then on every simply connected open set the induced $\Spin^c$-connection can be identified with the usual spinorial Levi-Civita connection after choosing a local unitary $A$-parallel section of $\mathbb{L}$.

The natural Hermitian product on $\Sigma_g^c M$ is denoted by $\langle \cdot,\cdot\rangle$ and is linear in the first argument.
In particular, for every vector field $X$ and every spinor field $\varphi$, the quantity $\langle X\cdot \varphi,\varphi\rangle$ is purely imaginary.

\subsection{Metric connections with vectorial torsion}\label{sec:vectorial}

We introduce metric connections with vectorial torsion because they provide the natural connection-theoretic framework for generalized imaginary $\Spin^c$-Killing spinors of type~I.
In later sections, these connections will allow us to reformulate the spinorial equation as the condition that a suitable spinor is parallel.
The relationship between the Levi-Civita connection $\nabla^g$ and another linear connection $\nabla$ can be described by a $(2,1)$-tensor field $S$, defined by
\[
    S_XY=\nabla^g_XY-\nabla_XY,\qquad \forall X,Y\in \mathfrak X(M).
\]

\begin{definition}
    Let $V$ be a vector field on $M$.
    A metric connection $\nabla=\nabla^g-S$ is called a \emph{vectorial connection associated with $V$} (or simply a \emph{$V$-connection}) if
    \begin{equation}\label{eq:cotorsion}
        S_XY=g(Y,V)\,X-g(X,Y)\,V,
    \end{equation}
    for all $X$, $Y\in \mathfrak X(M)$.
    A $V$-connection is said to be \emph{closed} (resp. \emph{exact}) if the $1$-form $V^\flat$ is closed (resp. exact).
\end{definition}

Next, we recall the induced action on the $\Spin^c$-spinor bundle.
Let $\nabla=\nabla^g-S$ be a metric connection on $TM$.
Since, for every $X\in \mathfrak X(M)$, the endomorphism $Y\mapsto S_XY$ is skew-symmetric with respect to $g$, it defines a $2$-form, still denoted by $S_X$.
We denote by $\nabla^A$ the induced $\Spin^c$-connection associated with $\nabla$ and the auxiliary connection $A$.
Then, for all $X\in \mathfrak X(M)$ and $\psi\in \Gamma(\Sigma_g^c M)$, we have
\begin{equation}\label{eq:spinorial-difference}
    \nabla^{g,A}_X\psi=\nabla^A _X\psi+\frac12\,S_X\cdot \psi.
\end{equation}
Thus, in the case of a $V$-connection, for all $X\in \mathfrak X(M)$ and $\psi\in \Gamma(\Sigma_g^c M)$, we obtain
\begin{equation}\label{eq:vectorial-spinor}
    \nabla^A _X\psi = \nabla^{g,A}_X\psi +\frac12\bigl(X\cdot V \cdot \psi+g(V,X)\psi\bigr).
\end{equation}

Our notation follows~\cite{AK2016}.
However, that general framework is broader than needed here, so we restrict attention to the class of vectorial connections introduced below.
These connections arise naturally in the study of generalized imaginary $\Spin^c$-Killing spinors; see~\ref{ex:hyperbolic-type-I}, \ref{ex:type-II-noncomplete}, or more generally~\ref{prop:principal-1}.
Moreover, if $\nabla$ is a $V$-connection admitting a parallel $\Spin^c$-spinor $\psi$ whose Dirac current is proportional to $V$, then $\nabla$ is an $(\alpha,\xi)$-connection for some function $\alpha$ and some unit vector field $\xi$; see~\ref{thm:Parallel to Imaginary}.

\begin{definition}\label{def:nabla}
    Let $\xi\in \mathfrak X(M)$ be a unit vector field and let $\alpha\in \mathcal C^\infty(M)$ be a smooth function that is not identically zero.
    A $V$-connection $\nabla$ is called an \emph{$(\alpha,\xi)$-connection} if $V=\alpha \xi$.
    Moreover, the connection is said to be \emph{$\xi$-parallel} if $\nabla \xi=0$.
\end{definition}

Note that this definition allows $\alpha$ to vanish at some points, as long as the vector field $\xi$ is globally defined. 
Given a unit vector field $\xi$, we denote by
\[
\mathcal{D}=\ker(\xi^\flat)=\{X\in TM: g(X,\xi)=0\}
\]
the orthogonal distribution to $\xi$.
We now study the consequences for $\mathcal{D}$ of the existence of a $\xi$-parallel $(\alpha,\xi)$-connection.
Since $\nabla \xi = 0$ and the endomorphism $X \mapsto S_X \xi$ is symmetric, the $1$-form $\xi^\flat$ is closed.
Consequently, by the Frobenius theorem $\mathcal{D}$ is integrable.
Moreover, the form $\alpha\xi^\flat$ is closed if and only if $d\alpha \wedge \xi^\flat=0$.
Furthermore, because $\xi$ is $\nabla$-parallel and $\nabla$ is metric, $\mathcal{D}$ is a $\nabla$-parallel distribution.

We state the corresponding formulas for the curvature tensor, Ricci curvature and scalar curvature of a $\xi$-parallel $(\alpha,\xi)$-connection in terms of the Levi-Civita connection.
We use the following conventions for the curvature tensor and the Ricci tensor of a connection $\nabla$:
\begin{equation}\label{eq:curv}
    \begin{alignedat}{2}
        R_{XY}Z \;&=\; \nabla_{[X,Y]}Z-\nabla_X(\nabla_YZ)+\nabla_Y(\nabla_XZ),\\
        \Ric(X,Y)\;&=\; -\sum_{i=1}^n g\bigl(R_{e_iX}Y,e_i\bigr),
    \end{alignedat}
\end{equation}
for all $X$, $Y$, $Z\in \mathfrak X(M)$, where $\{e_1,\dots,e_n\}$ denotes any local orthonormal frame.

\begin{proposition}\label{P:Formulas-Curvature}
    Let $(M,g)$ be equipped with a $\xi$-parallel $(\alpha,\xi)$-connection $\nabla$.
    Let $R$ and $R^g$ denote the curvature tensors of $\nabla$ and $\nabla^g$, respectively, and let $\Ric$, $\Ric^g$ and $\kappa$, $\kappa^g$ denote the corresponding Ricci and scalar curvatures.
    Then, for all $X$, $Y$, $Z\in \mathfrak X(M)$, the following identities hold:
    \begin{align*}
        R_{XY}Z     &=   R^g_{XY}Z+\alpha^2\bigl(g(X,Z)Y-g(Y,Z)X\bigr)                                                \\
                    &\quad +X(\alpha)\bigl(g(Z,\xi)Y-g(Y,Z)\xi\bigr) -Y(\alpha)\bigl(g(Z,\xi)X-g(X,Z)\xi\bigr),\\
        \Ric(X,Y)   &=  \Ric^g(X,Y)+\big( (n-1)\alpha^2 + \xi(\alpha) \big)g(X,Y) +(n-2)X(\alpha)g(Y,\xi),            \\
        \kappa      &= \kappa^g+n(n-1)\alpha^2+2(n-1)\xi(\alpha).
    \end{align*}
\end{proposition}

\begin{remark}
    These curvature identities follow from a direct computation based on the relation between $\nabla$ and $\nabla^g$ and the explicit expression of the difference tensor~$S$.
\end{remark}

\subsection{Weyl structures}\label{sec:conformal}
In this subsection, we relate vectorial connections to closed Weyl structures and explain how, locally, they arise from conformal changes of the metric.
This observation was already noted in~\cite[Rmk.~3.2]{AK2016}; we recall it here because it will later allow us to pass between parallel $\Spin^c$-spinors with respect to vectorial connections and generalized imaginary $\Spin^c$-Killing spinors on conformally related metrics.

Let $[g]$ be the conformal class of $g$ on $M$. 
A \emph{Weyl structure} on $(M, [g])$ is a torsion-free linear connection $\nabla ^{W}$ such that for all $h \in [g]$, there exists a $1$-form $\theta_h$ on $M$, called the \emph{Lee form} of $\nabla^W$ with respect to $h$, satisfying $\nabla^{W} h = 2\, \theta_h \otimes h$.
In other words, the Weyl structure preserves the conformal class.
After fixing $h\in [g]$, the Weyl structure is explicitly given, for all $X$, $Y\in \mathfrak{X}(M)$, by
\begin{align*}
    \nabla ^{W} _X Y = \nabla ^{h} _X Y -  \theta _h (Y) \,X + h(X,Y)\,\theta_h ^{\sharp} - \theta_h (X) \,Y, 
\end{align*}
where $\nabla^h$ is the Levi-Civita connection on $(M,h)$.
The first three terms are precisely the vectorial connection associated with $\theta_h ^\sharp$.
Thus, we may write $\nabla ^{W} _X Y = \nabla ^{\theta, h} _X Y - \theta_h (X)\, Y$, for all $X$, $Y\in \mathfrak{X}(M)$, where $\nabla^{\theta, h} _X Y = \nabla^{h} _X Y - \theta _h (Y)\, X + h(X,Y)\,\theta_h^{\sharp}$ is the vectorial connection associated with the vector field $\theta_h^{\sharp}$.
In particular, if we fix $h = g$, $V = \theta_g^{\sharp}$ and denote by $\nabla$ the $V$-connection, we have $\nabla ^{W} _X Y = \nabla _X Y - g(X, V)\, Y$, for all $X$, $Y\in \mathfrak{X}(M)$, or equivalently,
\begin{equation*}
    \nabla ^W  = \nabla  - V^\flat \otimes \Id.
\end{equation*}
Furthermore, we say the Weyl structure is \emph{closed} (resp.~\emph{exact}) if $V^\flat$ is closed (resp.~exact).

We now suppose that $V^\flat$ is closed.
Then there exists an open neighborhood $U \subset M$ such that $ V^\flat = df$, for some $f\in \mathcal{C}^{\infty} (U)$.
Let $K$ be an $(r,s)$-tensor field on $U$ and let $\nabla^W$ be the Weyl connection.
Then, for $X\in \mathfrak{X}(U)$, the covariant derivatives of $\nabla ^{W}$ and $\nabla$ are related as follows
\begin{equation} \label{eq:Symmetries}
\begin{split}
    \nabla^{W} _X (e^{(r-s) f}\, K) &= \nabla _X (e^{(r-s) f} \,K) - g(X,V)\, e^{(r-s)f }\, \Id \cdot K                                 \\
                                    &= (r -s)\, df (X)\, e^{(r-s) f} \,K + e^{(r-s) f}\, \nabla _X K + (s-r)\, g(X,V)\, e^{(r-s)f }\, K \\
                                    &= e^{(r-s) f}\,\nabla_X K,
\end{split}
\end{equation}
where ``$\cdot$" denotes the action of endomorphisms on tensors.
Indeed, \eqref{eq:Symmetries} shows that $\nabla ^W$ is the Levi-Civita connection of the conformally related metric defined below.
Furthermore, for all $X$, $Y$, $Z\in\mathfrak{X}(U)$, since $\nabla ^W_X (e^f Y) = e ^f \nabla _X Y$ and by the multilinearity of tensors, the curvature of $\nabla ^W$ and $\nabla$ are related as follows,
\begin{equation}\label{eq:Misma-Curvatura}
    R^W_{XY}Z = R_{XY}Z.
\end{equation}

We now discuss the corresponding relation between $\nabla^W$ and $\nabla$ on spinor fields, still under the assumption that $V^\flat$ is closed.
On the neighborhood $U$, the conformally related metric $\bar{g}= e^{-2f} g$ has $\nabla ^W$ as its Levi-Civita connection, so we write $\nabla^{\bar{g}}$ instead of $\nabla ^W$.
We can consider the two $\Spin^c$-spinor bundles $\Sigma _g ^c (U)$ and $\Sigma _{\bar{g}} ^c (U)$, together with the following identification:
\begin{equation}\label{eq:identif}
\begin{split}
    \overline{\cdot} \colon \Sigma_g^c (U)  &\to \Sigma_{\bar{g}}^c (U) \\
    \varphi = [u,\psi]                      &\longmapsto [e^{f} u,\psi].
\end{split}
\end{equation}
Let $\varphi \in \Sigma_g^c (U)$.
Using~\eqref{eq:identif}, we obtain that (see~\cite[Prop.~2.33]{BHMMM15}), for all $X \in \mathfrak{X}(U)$,
\[
\nabla ^{\bar{g},A} _X\bar{\varphi} = \overline{ \nabla ^{g,A} _X \varphi + \tfrac{1}{2} \, (X \cdot V \cdot \varphi + g(X,V) \varphi)}.
\]
Since $\nabla ^{g,A} _X \varphi = \nabla^A _X \varphi - \tfrac{1}{2} (X \cdot V \cdot \varphi + g(X,V) \varphi)$, we have
\begin{equation}\label{Eq:Weyl=vectorial}
    \nabla ^{\bar{g},A} _X\overline{\varphi} = \overline{ \nabla ^{A} _X  \varphi},
\end{equation}
for all $X \in \mathfrak{X}(U)$.
This spinorial covariant derivative is known as the weight-zero covariant derivative; see~\cite[Sec.~2.2.3]{BHMMM15}.

\section{Motivation and examples}\label{sec:imaginary-killing}

Here we follow the definitions and notation introduced in Section~\ref{sec:preliminaries}.
A spinor field $\psi \in \Gamma(\Sigma_g^c M)$ is called a \emph{generalized $\Spin^c$-Killing spinor} if there exists a complex-valued function $\lambda$ such that the following differential equation holds:
\begin{equation*}
    \nabla ^{g,A} _X \psi = \lambda X \cdot \psi,
\end{equation*}
for all $X \in \mathfrak{X}(M)$.
We refer to $\lambda$ as the \emph{Killing function} of $\psi$ or \emph{Killing number} of $\psi$ if $\lambda$ is constant.
When $\lambda$ is purely imaginary, we say $\psi$ is a \emph{generalized imaginary $\Spin^c$-Killing spinor}.

\subsection{Complete spin Riemannian manifolds with imaginary Killing spinors}

We briefly recall the complete spin case in order to place our results in context and to highlight the contrast with the incomplete setting.
By~\cite[p.~31, Thm.~8]{BFGK1991}, if $(M, g)$ admits a nonzero Killing spinor with Killing number $\lambda$, then $(M, g)$ is Einstein with scalar curvature $\kappa ^g = 4n(n-1)\lambda^2$.
Since the scalar curvature is real, a nonzero Killing spinor is called \emph{imaginary} if $\lambda \in i \mathbb{R}$, and \emph{real} if $\lambda \in \mathbb{R}$.

Imaginary Killing spinors, characterized by $\lambda = i \mu$ with $\mu \in \mathbb{R} \smallsetminus \{0\}$, exhibit distinct properties compared to real ones.
If $(M, g)$ is complete and admits an imaginary Killing spinor $\psi$, then $M^n$ must be non-compact; see~\cite[p.~31, Thm.~9]{BFGK1991}.
Moreover, the function $|\psi|^2$ is non-constant and never vanishes; see~\cite[p.~156, Lem.~1:(3)]{BFGK1991}.
A direct consequence is that imaginary Killing spinors cannot be parallel for any $g$-metric connection.

In spin geometry, complete Riemannian manifolds with imaginary Killing spinors exhibit specific geometric structures.
A complete, connected spin manifold $(M, g)$ admits imaginary Killing spinors with Killing number $i\mu$ if and only if $(M, g)$ is isometric to a warped product of the form $( \R\times F, dt^2 + e^{4\mu t} h )$, where $(F, h)$ is a complete connected spin manifold admitting a nonzero parallel spinor; see~\cite{B1989} for odd dimension and~\cite{B1989*} for the general case.
When the Killing function is not constant, Rademacher~\cite{R1991} classifies complete Riemannian manifolds either as warped products or as mapping tori over complete Riemannian manifolds admitting parallel spinors.
In~\ref{ex:hyperbolic-type-I}, we describe the construction of imaginary Killing spinors of type~I on the warped product described in~\cite{B1989}.

\subsection{Basic properties and motivation}
Let $\psi$ be a generalized imaginary $\Spin^c$-Killing spinor with Killing function $i\mu$.
There exists a vector field $V$, called the \emph{Dirac current} of $\psi$, such that $g(V,X) = i \langle X\cdot \psi, \psi \rangle$ for all vector fields $X$.
Whenever $V\neq 0$, we call the vector field $\xi = \frac{V}{|V|}$ \emph{the normalized Dirac current}.
A direct computation shows that, for all $X\in \mathfrak{X}(M)$, the following identities hold:
\begin{equation}\label{eq:conformal}
    X \left(|\psi|^2\right) = 2\mu V^\flat(X),\quad \nabla^g_X V = 2\mu |\psi|^2 X,\quad X \left(|V|^2\right) = 4\mu|\psi|^2 V^\flat(X).
\end{equation}
Since $\nabla^g V$ is symmetric, we have $dV^\flat = 0$ and $\mathcal{L}_V g = 2 \mathrm{Sym} (\nabla^g V^\flat) = 4 \mu |\psi|^2 g$; thus $V$ is a closed conformal vector field.
Two direct consequences of~\eqref{eq:conformal} are that the normalized Dirac current $\xi$ is geodesic whenever it is defined and that $q_{\psi} = |\psi|^4 - |V|^2$ is a constant.
Furthermore, this constant is non-negative because the Cauchy--Schwarz inequality implies that $|V|^2 = i \langle V \cdot \psi, \psi \rangle \leq |V|\,|\psi|^2$.
The alternatives $q_\psi=0$ and $q_\psi>0$ divide generalized imaginary $\Spin^c$-Killing spinors into two types: $\psi$ is of \emph{type~I} if $q_{\psi} = 0$, and of \emph{type~II} if $q_{\psi} > 0$.
Recall that generalized imaginary $\Spin^c$-Killing spinors never vanish; see~\cite[Lem.~3.1]{GN2015}.
In particular, the Dirac current $V$ of a type~I spinor never vanishes.

On the one hand, if $q_{\psi}=0$, then $|V|=|\psi|^2$ and
\[
    |V|=g(V,\xi)= i\langle \xi\cdot\psi,\psi\rangle \le |\xi\cdot\psi|\,|\psi|=|\psi|^2=|V|,
\]
so equality holds in Cauchy--Schwarz, whence $\xi\cdot\psi=\lambda\psi$ for some $|\lambda|=1$; since $\xi\cdot\xi\cdot\psi=-\psi$, we get $\lambda=\pm i$, and the identity $g(V,\xi)=|V|$ fixes the sign to $\lambda=-i$.
Conversely, if $\xi\cdot\psi=-i\psi$, then $|V|=-i\langle \xi\cdot\psi,\psi\rangle=|\psi|^2$ and $q_{\psi}=0$.
Therefore, a generalized imaginary $\Spin^c$-Killing spinor $\psi$ is of type~I if and only if its normalized Dirac current $\xi$ satisfies $\xi \cdot \psi = - i \psi$.

The following example shows that exact $\xi$-parallel $(\alpha,\xi)$-connections arise naturally in the context of type~I imaginary Killing spinors.
This will serve as the model for the correspondence established later in Section~\ref{sec:type-I}.

\begin{example}\label{ex:hyperbolic-type-I}
    Let $(F,h)$ be a connected spin manifold carrying a non-trivial parallel spinor field $\phi$, and let $\mu>0$ be a constant.
    We consider the warped product $(M,g):=(\mathbb{R}\times F,\ dt^2+e^{4\mu t}h)$.
    If $(F,h)=(\mathbb{R}^{2m},g_{\mathbb{R}^{2m}})$, then $(M,g)$ is isometric to the real hyperbolic space $\mathbb{H}^{2m+1}$ with sectional curvature $-4\mu^2$.
    We divide the example into two steps.
    First, we show that $(M,g)$ carries a natural exact $\xi$-parallel $(\alpha,\xi)$-connection $\nabla$ for some unit vector field $\xi$.
    Then, we use this connection to construct the imaginary Killing spinor.

    First, we consider $s:=\frac{1}{2\mu}e^{-2\mu t}$.
    Then $ds=-e^{-2\mu t}\,dt$, and $dt^2=e^{4\mu t}ds^2$; consequently, the metric can be written as
    \[
        g= e^{4\mu t} \left(ds^2+h\right).
    \]
    Therefore, if we set $\bar{g}:=ds^2+h$, then $\bar{g}=e^{-4\mu t}g=e^{-2f}g$, where $f=2\mu t=-\ln(2\mu s)$.
    In the notation of Section~\ref{sec:conformal}, the corresponding Lee form is $V^\flat=df=2\mu\,dt$. 
    Let $\xi:=\frac{\partial}{\partial t}$, so $V=2\mu\,\xi$. 
    Then the $(2\mu, \xi)$-connection $\nabla=\nabla^g-S$ is exact.
    Since $\frac{\partial}{\partial s}$ is $\nabla^{\bar{g}}$-parallel, by~\eqref{eq:Symmetries} the vector field $\xi$ is $\nabla$-parallel.
    Hence, $\nabla$ is an exact $\xi$-parallel $(2\mu,\xi)$-connection.

    Second, we extend the parallel spinor $\phi$ to a parallel spinor $\bar\varphi(t,x) = \phi(x)$ on $(M,\bar{g})$.
    Since $\partial_s$ is $\nabla^{\bar{g}}$-parallel, the spinor $\bar\varphi$ splits into two parallel spinors $\bar\varphi^\pm$ satisfying $\partial_s\cdot_{\bar{g}}\bar\varphi^\pm=\mp i\,\bar\varphi^\pm$.
    Without loss of generality, we can assume that $\partial_s\cdot_{\bar{g}}\bar\varphi=i\,\bar\varphi$; the other case is analogous.
    By the identification~\eqref{eq:identif}, we can identify $\bar\varphi \in \Sigma_{\bar{g}}^c$ with a spinor field $\varphi \in \Sigma_{g}^c$, and because~\eqref{Eq:Weyl=vectorial} holds and $\xi =- e^{-f} \frac{\partial}{\partial s}$, the spinor $\varphi$ satisfies $\xi \cdot_g \varphi = - i \varphi$ and
    \[
        \nabla^{g,A_0} _X \varphi = \nabla^{A_0}_X \varphi - \mu (X \cdot \xi \cdot \varphi + g(X,\xi) \varphi) = \mu i(X \cdot \varphi - g(X,\xi) \xi\cdot \varphi).
    \]
    Hence, the spinor $\psi = e^{\mu t} \varphi$ is a type~I imaginary Killing spinor with Killing number $i\mu$. 
\end{example}

The previous example illustrates the geometric mechanism behind type~I spinors: they arise naturally from exact $\xi$-parallel $(\alpha,\xi)$-connections.
We now turn to type~II, where the complete theory is much more rigid, although incomplete examples still exist.

If a connected complete spin Riemannian manifold carries a non-trivial imaginary Killing spinor $\psi$ of type~II with Killing number $i\mu$, then $(M,g)$ is isometric to the real hyperbolic space of constant sectional curvature $-4\mu^2$, see~\cite{R1991}.
Analogously, in the $\Spin^c$ setting, if the dimension of the manifold is $n\geq 3$, then the same conclusion holds; see~\cite{GN2015}.
By contrast, when $n=2$,~\cite[Sec.~5]{GN2015} shows that there exist examples with non-constant Killing number.
In particular, the only complete $\Spin^c$-Riemannian manifolds admitting imaginary $\Spin^c$-Killing spinors of type~II with constant Killing number are the real hyperbolic spaces.
Nevertheless, incomplete warped-product examples of generalized imaginary $\Spin^c$-Killing spinors of type~II exist; see the following example.

\begin{example}\label{ex:type-II-noncomplete}
    Let $(F,h)$ be a connected $\Spin^c$-Riemannian manifold with auxiliary connection $A^F$. Suppose that $(F,h,A^F)$ carries two non-trivial real $\Spin^c$-Killing spinors $\Phi ^\pm$ with Killing numbers $\pm\mu$, where $\mu>0$.
    On
    \[
        M=(0,\infty)\times F,\qquad g=dt^2+\sinh^2(2\mu t)\,h,
    \]
    we take the induced $\Spin^c$-structure and the pull-back auxiliary connection $A=\pi_F^*A^F$. Note that the metric $g$ is incomplete.
    Then the spin construction of~\cite[pp.~160--161]{BFGK1991} extends verbatim to this setting and produces a generalized imaginary $\Spin^c$-Killing spinor with Killing number $i\mu$.
    In particular, in the spin case, if $(F,h)$ is not locally isometric to the round sphere, then the resulting spinor is of type~II; see~\cite[pp.~160--161]{BFGK1991}.
\end{example}

\section{\texorpdfstring{On generalized imaginary Spin$^c$-Killing spinors}{On generalized imaginary Spin-c-Killing spinors}} \label{sec:main}

In this section, we combine the identities derived in Section~\ref{sec:imaginary-killing} with the vectorial-connection viewpoint from Section~\ref{sec:preliminaries} in order to describe the geometry of $\Spin^c$-Riemannian manifolds carrying generalized imaginary $\Spin^c$-Killing spinors.
We first treat the case where the Dirac current never vanishes (subsection~\ref{sec:charact}), then the case where it has zeros (subsection~\ref{sec:vanish}), and we obtain a further refinement under an irreducibility assumption on the orthogonal distribution $\mathcal{D}$ (subsection~\ref{sec:irreducible}).

Throughout this section and the results that follow, we assume the following hypothesis:
\begin{enumerate}[label=\textup{(H\arabic*)},ref=H\arabic*]
    \item\label{assumption:H1} 
    $(M,g)$ is a $\Spin^c$-Riemannian manifold that admits a generalized imaginary $\Spin^c$-Killing spinor $\psi$ with Dirac current $V$ and Killing function $i\mu$. 
\end{enumerate}
Furthermore, we use the following notation, whenever it is defined,
\begin{equation*}
    \alpha = 2\mu \frac{|\psi|^2}{|V|},                 \quad
    \xi = \frac{V}{|V|},                                \quad
    \mathcal{D} = \ker(\xi^\flat),                      \quad
    \mathcal{L} = \operatorname{span}_{\mathbb{R}} \xi, \quad
    q_{\psi} = |\psi|^4 - |V|^2.
\end{equation*}

\subsection{When the Dirac current never vanishes}\label{sec:charact}
When the Dirac current is nowhere vanishing, the spinor determines a natural geometric structure adapted to the normalized Dirac current.
We first show that the associated $(\alpha,\xi)$-connection is exact and satisfies $\nabla\xi=0$; in particular, the orthogonal distribution $\mathcal D=\ker(\xi^\flat)$ is integrable and $\nabla$-parallel.
We then analyze the spinorial geometry induced on the leaves of $\mathcal D$.
These two ingredients are combined with a conformal splitting argument to obtain the simply connected global models.
Finally, we study the action of the deck transformation group and derive the global classifications in type~I and type~II.

\begin{proposition}\label{prop:principal-1}
    Under hypothesis~\eqref{assumption:H1}, if $V$ never vanishes, then the $(\alpha, \xi)$-connection $\nabla$ is exact and satisfies $\nabla \xi = 0$.
    Moreover, $\mathcal{D}$ is integrable and a $\nabla$-parallel distribution, and $\alpha$ is constant along the leaves of $\mathcal{D}$.
\end{proposition}
\begin{proof}
    Let $X \in \mathfrak{X}(M)$.
    We first compute the covariant derivative of the normalized Dirac current $\xi$.
    By~\eqref{eq:conformal}, we obtain $X(|V|) = \alpha V^\flat(X)$ and, using this equality, we compute
    \begin{equation}\label{eq:derivada-xi}
        \nabla^g_X \xi = \alpha \bigl(X - g(X,\xi)\, \xi\bigr).
    \end{equation}
    We define the $(\alpha, \xi)$-connection $\nabla$ by $\nabla = \nabla^g - S$, where $S$ is the vectorial endomorphism field associated with $\alpha \xi$.
    We check that, by definition, $\nabla \xi = 0$ and consequently, the orthogonal distribution $\mathcal{D}$ is $\nabla$-parallel as well.
    Since $\nabla^g \xi$ is symmetric, we have $\xi ^\flat$ is closed and thus, by the Frobenius theorem, $\mathcal{D}$ is integrable. 
    
    We check that $\alpha \xi^\flat$ is exact.
    Using again that $X(|V|) = \alpha V^\flat (X)$, we divide by $|V|$ and obtain
    \[
        \frac{X(|V|)}{|V|} = \alpha \xi^\flat(X).
    \]
    Hence, $\alpha \xi^\flat = d(\ln |V|)$ is exact.
    Furthermore, we have $d\alpha \wedge \xi^\flat = 0$, and thus $\alpha$ is constant along the leaves of $\mathcal{D}$.
\end{proof}

Once the exact $(\alpha, \xi)$-connection has been identified, the next step is to understand the geometry induced on the leaves of the orthogonal distribution $\mathcal{D}$.

\begin{proposition}\label{prop:leaves}
    Under hypothesis~\eqref{assumption:H1}, if $V$ never vanishes, then:
    \begin{itemize}
        \item if $\psi$ is of type~I, the leaves of $\mathcal{D}$ admit a non-trivial parallel $\Spin^c$-spinor;
        \item if $\psi$ is of type~II, the leaves of $\mathcal{D}$ admit at least two non-trivial real $\Spin^c$-Killing spinors with opposite Killing numbers $\pm \mu \frac{\sqrt{q_\psi}}{|V|}$, where $q_{\psi} = |\psi|^4 - |V|^2$.
    \end{itemize}
\end{proposition}

\begin{proof}
    Since $\nabla \xi = 0$ and $(i\xi\cdot)^2 = 1$, the endomorphism $i\xi\cdot$ splits $\Sigma_g^c M$ into two orthogonal $\nabla^A$-parallel subbundles,
    \[
        \Sigma_g^c M = \Sigma^+ M \oplus \Sigma^- M,
    \]
    where $\Sigma^\pm M$ are the eigensubbundles associated with the eigenvalues $\pm 1$.
    Let $\psi^+$ and $\psi^-$ denote the projections of $\psi$ onto $\Sigma^+ M$ and $\Sigma^- M$, respectively.
    We now compute their $\nabla^A$-covariant derivatives. 
    Recall that $\nabla = \nabla^g - S$.
    The skew-symmetric endomorphism $Y \mapsto S_X Y$ corresponds to the $2$-form $S_X = - \alpha (X\wedge \xi)$.
    Hence
    \[
        \frac{1}{2} S_X \cdot = - \frac{\alpha}{2} (X\cdot \xi \cdot + g(X,\xi)).
    \]
    Thus,
    \begin{align*}
        \nabla^A_X \psi &= \nabla^{g,A}_X \psi - \frac{1}{2} S_X \cdot \psi = i\mu X \cdot \psi + \frac{\alpha}{2} (X\cdot \xi \cdot+ g(X,\xi)) \psi     \\
                        &= i\left(\mu - \frac{\alpha}{2}\right) X \cdot \psi^+ + i \left(\mu + \frac{\alpha}{2}\right) X\cdot \psi^- + \frac{\alpha}{2} g(X,\xi) \psi^+ + \frac{\alpha}{2} g(X,\xi) \psi^-,
    \end{align*}
    Writing $\bar{X} = X - g(X,\xi) \xi$, we decompose $X = \bar{X} + g(X,\xi) \xi$.
    Since $\bar{X}$ and $\xi$ are orthogonal,
    \begin{align*}
        (i\xi)\cdot (\bar{X}\cdot \psi^\pm) &= - \bar{X} \cdot (i\xi)\cdot \psi^\pm = \mp \bar{X} \cdot \psi^\pm.
    \end{align*}
    In other words, $\bar{X} \cdot $ is an endomorphism that interchanges the eigensubbundles $\Sigma^+ M$ and $\Sigma^- M$.
    Therefore, 
    \begin{equation*}
    \begin{alignedat}{111}
        \nabla^A_X \psi^+   
        &= i\left(\mu + \frac{\alpha}{2}\right) \bar{X} \cdot \psi^- + \left(\mu - \frac{\alpha}{2}\right) g(X,\xi) \psi^+ + \frac{\alpha}{2}g(X,\xi) \psi^+ \\          
        &= i\left(\mu + \frac{\alpha}{2}\right) \bar{X} \cdot \psi^- + \mu g(X,\xi) \psi^+ .
    \end{alignedat}
    \end{equation*}
    and, analogously,
    \begin{equation*}
        \nabla^A_X \psi^- = i\left(\mu - \frac{\alpha}{2}\right) \bar{X} \cdot \psi^+ - \mu g(X,\xi) \psi^-.
    \end{equation*}

    Next, we compute the variation of the norms of $\psi^+$ and $\psi^-$:
    \begin{align*}
        X \left(|\psi^+|^2\right)
        &= 2 \mathrm{Re} \langle \nabla^A_X \psi^+, \psi^+ \rangle \\
        &= 2 \mathrm{Re} \left\langle i\left(\mu + \frac{\alpha}{2}\right) \bar{X} \cdot \psi^- + \mu g(X,\xi) \psi^+, \psi^+ \right\rangle \\
        &= 2 \left(\mu + \frac{\alpha}{2}\right) \mathrm{Re} \left\langle i \bar{X} \cdot \psi^- , \psi^+ \right\rangle + 2\mu g(X,\xi) |\psi^+|^2
    \end{align*}
    and, analogously,
    \begin{align*}
        X \left(|\psi^-|^2\right) 
        &= 2\left(\mu - \frac{\alpha}{2}\right) \mathrm{Re} \left\langle i\bar{X} \cdot \psi^+, \psi^- \right\rangle - 2\mu g(X,\xi) |\psi^-|^2.
    \end{align*}
    Set $\rho_{\pm} = |\psi^{\pm}|$.
    Since $\psi^+$ and $\psi^-$ are orthogonal,
    \[
        |\psi|^2 = \rho_+^2 + \rho_-^2, \qquad |V| = \rho_+^2 - \rho_-^2.
    \]
    Since $\bar{X}(|\psi|^2) = 0$ and $\bar{X}(|V|^2) = 0$, it follows that $\bar{X}(\rho_+ ^2) = \bar{X}(\rho_- ^2) = 0$. 

    We now distinguish between the two types of spinors.
    If $\psi$ is of type~I, then, by our convention, $\psi = \psi^+$ and $\xi \cdot \psi = - i \psi$.
    We therefore consider the spinor field $\phi = \frac{\psi}{|\psi|}$ and compute its $\nabla^A$-covariant derivative: 
    \begin{equation}\label{eq:spinor-paralelo}
        \nabla^A _X \phi =  \mu g(X,\xi) \phi - \frac{X(|\psi|)}{|\psi|} \phi =  \mu g(X,\xi) \phi - \frac{\mu |V|}{|\psi|^2}g(X,\xi) \phi = 0.
    \end{equation}
    If $\psi$ is of type~II, then $\rho_+$ and $\rho_-$ are both nonzero.     
    Hence $X (\rho_+) = \mu \rho_+ g(X,\xi)$ and $X (\rho_-) = - \mu \rho_- g(X,\xi)$.
    In particular, $X (\rho_+ \rho_-) = 0$, so we may define the constant $c = \rho_+\rho_-$.
    Note that $4c^2 = q_{\psi}$.
    We have
    \begin{align*}
        \mu - \frac{\alpha}{2} &= \mu - \mu \frac{|\psi|^2}{|V|} = \mu - \mu \frac{\rho_+^2 + \rho_-^2}{\rho_+^2 - \rho_-^2} = - 2\mu \frac{\rho_-^2}{\rho_+^2 - \rho_-^2} = - 2\mu \frac{\rho_-^2}{|V|} , \\
        \mu + \frac{\alpha}{2} &= \mu + \mu \frac{|\psi|^2}{|V|} = \mu + \mu \frac{\rho_+^2 + \rho_-^2}{\rho_+^2 - \rho_-^2} = 2\mu \frac{\rho_+^2}{\rho_+^2 - \rho_-^2} =  2\mu \frac{\rho_+^2}{|V|}.
    \end{align*}
    Then, we can define $\phi^\pm = \rho_\mp \psi^\pm$ and compute
    \begin{align*}
        \nabla^A _X \phi^+ 
        &= X(\rho_-)\psi^+ + \rho_- \nabla^A _X \psi^+ \\
        &= -\mu \rho_- g(X,\xi)\psi^+ + \rho_-\left( 2\mu i \frac{\rho_+ ^2}{|V|} \bar X\cdot \psi^- + \mu g(X,\xi)\psi^+ \right)  = 2i\mu \frac{c}{|V|}\,\bar X\cdot \phi^-,
    \end{align*}
    and, analogously,
    \[
        \nabla^A _X \phi^- = -2i\mu \frac{c}{|V|}\,\bar X\cdot \phi^+ .
    \]
    We now define $\Phi^\pm = \phi^+ \pm \phi^-$.
    A direct computation shows that
    \begin{equation}\label{eq:spinors-total}
    \begin{alignedat}{111}
        \nabla^A _X \Phi^+  &= -2i\mu \frac{c}{|V|}\,\bar X\cdot \Phi^-, \qquad \nabla^A _X \Phi^-
                            &= 2i\mu \frac{c}{|V|}\,\bar X\cdot \Phi^+.
    \end{alignedat}
    \end{equation}

    It remains to translate these ambient spinorial equations into equations on the leaves of $\mathcal D$.
    Let $F$ be a leaf of $\mathcal{D}$ and let $\nabla^\top$ be the $\Spin^c$-connection associated with the Levi-Civita connection and the connection $A$.
    We now compute the covariant derivative of the projection of $\Phi^\pm$ onto $F$.
    Let $\iota \colon F \to M$ be the inclusion map.
    We recall that the spinorial connection $\nabla^\top$ is related to $\nabla^{g,A}$ by the following formula, see~\cite[Thm.~2.39]{BHMMM15},
    \[
        \nabla^\top_{\bar X} \iota^* \varphi = \iota^* \left(\nabla^{g,A}_{\bar X} \varphi\right) + \frac{1}{2} \iota^* \left(  \left(\nabla^g_{\bar X} \xi\right) \cdot \xi \cdot \varphi \right),
    \]
    for all $\bar X \in \mathfrak{X}(F)$ and $\varphi \in \Sigma_g^c M$.
    Furthermore, the Clifford multiplication on $F$ is denoted by $\bullet$ and is related to the Clifford multiplication on $M$ as follows (see~\cite[Lem.~2.38]{BHMMM15}):
    \[
        \bar X \bullet \iota^*\varphi= \iota^* (\bar X \cdot \xi \cdot \varphi),
    \]
    for all $\bar X \in \mathfrak{X}(F)$ and every $\varphi \in \Sigma_g^c M$.
    Using these two relations, we can compute the covariant derivative of the projection of $\Phi^\pm$ onto $F$.
    First, by~\eqref{eq:derivada-xi}, we have $\left(\nabla^g_{\bar X} \xi\right) \cdot \xi \cdot = - S_{\bar X} \cdot$, for all $\bar X\in \mathfrak{X}(F)$.
    Hence, we obtain that
    \[
        \nabla^\top_{\bar X} \iota^* \varphi =   \iota^* \left(\nabla^A _{\bar X} \varphi\right),
    \]
    for all $\bar X \in \mathfrak{X}(F)$ and $\varphi \in \Sigma_g^c M$.
    
    We now restrict the spinors constructed in the two cases to the leaves of $\mathcal D$.
    If $\psi$ is of type~I, then $\nabla^A \phi = 0$ and consequently, $\nabla^\top_{\bar X} \iota^* \phi = \iota^* \left( \nabla^A _{\bar X} \phi \right) = 0$. 
    If $\psi$ is of type~II, then we use~\eqref{eq:spinors-total} and $i\phi^\pm = \mp \xi \cdot \phi^\pm$ to obtain
    \begin{align*}
        \nabla^\top_{\bar X} \iota^* \Phi^+ \quad&= \iota^* \left( \nabla^A _{\bar X} \Phi^+ \right) = \iota^* \left( - 2i\mu \frac{c}{|V|}\,\bar X\cdot \Phi^- \right) \\
        &= - 2\mu \frac{c}{|V|}\, \iota^* \left( \bar X\cdot (i\phi^+ - i\phi^-) \right) \\
        &= 2\mu \frac{c}{|V|}\, \iota^* \left( \bar X\cdot \xi \cdot (\phi^+ + \phi^-) \right) =\quad  2\mu \frac{c}{|V|}\, \iota^* \left( \bar X\cdot \xi \cdot \Phi^+ \right).
    \end{align*}
    Thus, we obtain that 
    \[
        \nabla^\top_{\bar X} \iota^* \Phi^+ =  2\mu \frac{c}{|V|}\,  \bar X \bullet \iota^* \Phi^+, \quad \nabla^\top_{\bar X} \iota^* \Phi^- = -2\mu \frac{c}{|V|}\,  \bar X \bullet \iota^* \Phi^-.
    \]
    This proves the result.
\end{proof}

The previous proposition gives the leafwise spinorial geometry.
To turn this information into a global description of $M$, we now use a conformal splitting argument.
We use the following result, which is not explicitly stated in~\cite{PR1993} but follows directly from~\cite[Cor.~2]{PR1993}.

\begin{theorem}\label{thm:splitting}
    Let $(M,\bar{g})$ be a simply connected Riemannian manifold. 
    Assume that $M$ admits two complementary, orthogonal, integrable, and $\nabla^{\bar{g}}$-parallel distributions $\mathcal L$ and $\mathcal F$. 
    If one of the distributions is complete, then $(M,\bar{g})$ is isometric to a Riemannian product $(L,ds^2) \times (F,h)$, where $(L,ds^2)$ and $(F,h)$ are integrable Riemannian submanifolds of the foliations $\mathcal L$ and $\mathcal F$, respectively.
\end{theorem}

Our goal is to understand the geometry of the leaves of $\mathcal{D}$ and of the line field $\mathcal{L} = \mathrm{span}_\R \xi$.
We study both distributions with respect to the metric $\bar{g} = |V|^{-2} g$, or equivalently, with respect to the Levi-Civita connection $\nabla^{\bar{g}} = \nabla - \theta \otimes \Id$, where $\theta = \alpha \xi^\flat$.
Since $\mathcal{D}$ and $\mathcal{L}$ are $\nabla$-parallel and $\theta \in \mathcal{L}$, they are also $\nabla^{\bar{g}}$-parallel.
We also use~\cite[Cor.~1]{PR1993}.
For this, we need to verify that the leaves of $\mathcal{D}$ are spherical, which follows immediately from $\nabla^g _X \xi = \alpha X$ for all $X \in \mathcal{D}$ and $X(\alpha) = 0$ whenever $d\alpha \wedge \xi^\flat = 0$.

Apart from the simply-connectedness hypothesis and the completeness of one of the distributions, we are in the setting of the preceding theorem.
We now describe the geometry of $\Spin^c$-Riemannian manifolds carrying a generalized imaginary $\Spin^c$-Killing spinor with nowhere-vanishing Dirac current.
The remaining argument is organized around four steps:
\begin{enumerate}
    \item we study how $|V|$ varies along integral curves of $\xi$, in~\ref{lem:step-1};
    \item we show that, for $\dim M \geq 3$, completeness of $\xi$ and type~II are incompatible, in~\ref{lem:step-2};
    \item assuming $(M,g)$ is simply connected and either the leaves of $\mathcal{D}$ are complete or $\xi$ is complete, we describe the global geometry in the type~I and type~II cases, in~\ref{lem:step-3};
    \item finally, we analyze the deck transformation properties for the universal cover, in~\ref{lem:step-4}.
\end{enumerate}

\begin{lemma}\label{lem:step-1}
    Under hypothesis~\eqref{assumption:H1}, if $V$ never vanishes and $\gamma\colon I \to M$ is a maximal integral curve of $\xi$, then the function $u(t) = |V|(\gamma(t))$ satisfies:
    \begin{enumerate}
        \item if $\psi$ is of type~I, there exist constants $t_0$, $C \in \R$ such that
        \[
            u(t) = C e^{2\int_{t_0}^t \mu(s) ds}.
        \]
        \item if $\psi$ is of type~II, there exist constants $t_0$, $C \in \R$ such that
        \[
            u(t) = \sqrt{q_{\psi}}\,\sinh\!\left(2 \int_{t_0}^t \mu(s) ds + C\right).
        \]
    \end{enumerate}
\end{lemma}

\begin{proof}
    Let $\gamma:I\to M$ be a maximal integral curve of $\xi$, with $|\dot{\gamma}(t)|=1$ for all $t\in I$, and set $u(t):=|V|(\gamma(t))$. 
    Since $\xi=\frac{V}{|V|}$ and $X(|V|)=2\mu \frac{|\psi|^2}{|V|}V^\flat(X)$, taking $X=\xi$ yields
    \[
        u'=\xi(|V|)=2\mu\,|\psi|^2.
    \]
    On the other hand, the quantity $q_{\psi}:=|\psi|^4-|V|^2$ is constant on $M$, hence along $\gamma$ we have $|\psi|^2=\sqrt{u^2+q}$.
    Therefore
    \[
        u'=2\mu\,\sqrt{u^2+q}.
    \]
    Assume first that $\psi$ is of type~I.
    Then $q_{\psi}=0$, so the previous equation reduces to $u'(t)=2\mu(t) \,u(t)$.
    It follows that there exist $t_0$, $C \in \R$ such that $u(t)= C e^{2\int_{t_0}^{t}\mu(s)ds}$.
    Assume now that $\psi$ is of type~II.
    Then $q_{\psi}>0$, and dividing by $\sqrt{u^2+q_{\psi}}$, we obtain
    \[
        \frac{d}{dt}\left(\mathrm{arsinh}\!\left(\frac{u(t)}{\sqrt{q_{\psi}}}\right)\right)=2\mu(t).
    \]
    Hence there exist two constants $t_0$, $C \in \R$ such that
    \[
        u(t) =\sqrt{q_{\psi}}\,\sinh\!\left(2 \int_{t_0}^{t}\mu (s) ds + C\right),
    \]
    for all $t$ in the interval of definition of $\gamma$ for which the right-hand side is positive.
\end{proof}

The differential equation for $|V|$ obtained in the previous lemma already suggests a strong restriction in the type~II case.
In dimension at least three, this becomes decisive once one knows that the Killing function is constant.

\begin{lemma}\label{lem:step-2}
    Under hypothesis~\eqref{assumption:H1}, if $V$ never vanishes, $\psi$ is of type~II, and $\dim M \geq 3$, then $\mu$ is constant and $\xi$ cannot be complete.
\end{lemma}

\begin{proof}
    We first show that $\mu$ is constant and then use this to conclude that $\xi$ cannot be complete.
    For the first part, the argument is the same as in~\cite[Lem.~4.4, Prop.~4.5]{GN2015}. 
    Suppose, by contradiction, that $d\mu \neq 0$ at some point of $M$.
    Then there exists an open neighborhood $U \subset M$ on which $d\mu$ does not vanish.
    By~\cite[Lem.~4.4]{GN2015}, we have $\langle X \cdot \psi , \psi \rangle = 0$ for every vector field $X$ orthogonal to $\grad \mu$.
    Hence, on $U$, the Dirac current $V$ is pointwise parallel to $\grad \mu$.
    After possibly shrinking $U$, we may consider an orthonormal frame $\{e_1, \ldots, e_n = \frac{\grad \mu}{|\grad \mu|}\}$ on $U$.

    We now use the pointwise part of the proof of~\cite[Prop.~4.5]{GN2015}.
    Notice that no completeness is involved in the following calculation.
    The global completeness assumption in~\cite[Prop.~4.5]{GN2015} is used only to ensure the existence of a point at which $\langle d\mu \cdot \psi, \psi \rangle$ is nonzero; here this follows directly from $V\neq 0$.

    For $i=1,\ldots,n-1$, we apply~\cite[Lem.~3.1(i)]{GN2015}, Clifford-multiply it by $e_i$, and take the scalar product with $\psi$.
    Since $\langle e_i\cdot\psi,\psi\rangle=0$, we obtain 
    \begin{equation}\label{eq:GN-Lich-local}
        i(n-1)\langle e_i\cdot d\mu\cdot\psi,\psi\rangle = \frac{i}{2}\, \Omega^A(e_i,e_n)\, \langle e_n\cdot\psi,\psi\rangle.
    \end{equation}
    On the other hand, we apply the identity in~\cite[Lem.~3.1(iii)]{GN2015} with $X=e_i$, take the scalar product with $\psi$, and use again $\langle e_j\cdot\psi,\psi\rangle=0$ for $j=1,\ldots,n-1$, to obtain 
    \begin{equation}\label{eq:GN-Ricci-local}
        \frac12\Ric^g(e_i,e_n)\langle e_n\cdot\psi,\psi\rangle - \frac{i}{2}\Omega^A(e_i,e_n)\langle e_n\cdot\psi,\psi\rangle = i\langle d\mu\cdot e_i\cdot\psi,\psi\rangle.
    \end{equation}
    Taking the imaginary part of~\eqref{eq:GN-Ricci-local} yields $\Ric^g(e_i,e_n)\langle e_n\cdot\psi,\psi\rangle=0$ and since $\langle e_n\cdot\psi,\psi\rangle\neq0$, we have $\Ric^g (e_i,e_n)=0$.
    Taking the real part of~\eqref{eq:GN-Ricci-local}, we obtain  
    \begin{equation}\label{eq:GN-real-local}
        -\frac{i}{2}\Omega^A(e_i,e_n)\langle e_n\cdot\psi,\psi\rangle = i\langle d\mu\cdot e_i\cdot\psi,\psi\rangle .
    \end{equation}
    Finally, combining~\eqref{eq:GN-Lich-local} and~\eqref{eq:GN-real-local}, we get
    \[
        (n-1)\langle e_i\cdot d\mu\cdot\psi,\psi\rangle = \langle e_i\cdot d\mu\cdot\psi,\psi\rangle.
    \]
    As $n\geq3$, this implies $\langle e_i\cdot d\mu\cdot\psi,\psi\rangle=0$ and, by~\eqref{eq:GN-Lich-local}, $e_n \lrcorner \Omega^A = 0$. 

    We now use~\cite[Lem.~3.1(iii)]{GN2015} for $X = e_n$ and obtain that
    \[
        \left(\frac{1}{2}\Ric^g(e_n,e_n)+2(n-1)\mu^2\right)e_n\cdot\psi = i(n-1)|d\mu|\,\psi .
    \]
    In other words, $e_n \cdot \psi$ is pointwise proportional to $\psi$.
    Thus, the Cauchy--Schwarz inequality becomes an equality, and therefore, $\psi$ must be of type~I in contradiction to our assumption. This proves that $\mu$ is constant.

    Finally, since $\mu$ is constant and, by~\ref{lem:step-1}, $u(t)=\sqrt{q_{\psi}}\,\sinh(2\mu t+C)$, it follows that this function cannot be positive for all $t\in\R$.
    Hence, the vector field $\xi$ cannot be complete.
\end{proof}

We are now in a position to combine the conformal splitting with the leafwise classification from~\ref{prop:leaves}, thereby obtaining the simply connected global models.

\begin{lemma}\label{lem:step-3}
    Under hypothesis~\eqref{assumption:H1}, suppose that $V$ never vanishes, $(M,g)$ is simply connected, and either $\xi$ is complete or the leaves of $\mathcal D$ are complete.
    Then:
    \begin{enumerate}
        \item If $\psi$ is of type~I, there exists a connected open interval $J\subset \R$ and a Riemannian manifold $(F,h)$ admitting a non-trivial parallel $\Spin^c$-spinor such that, for some $t_0\in J$,
        \[
            (M,g) \cong \bigl(J\times F,\ dt^2+e^{4\int_{t_0}^t \mu(s)\,ds}\,h\bigr).
        \]

        \item If $\psi$ is of type~II and $\dim M\geq 3$, then $\mu$ is constant, $\xi$ is not complete, and there exist $t_m \in (0, \infty]$ and a Riemannian manifold $(F,h)$ admitting two non-trivial real $\Spin^c$-Killing spinors with opposite Killing numbers $\pm \mu$ such that
        \[
            (M,g)\cong \bigl((0,t_m)\times F,\ dt^2+\sinh^2(2\mu t)\,h\bigr).
        \]
    \end{enumerate}
\end{lemma}

\begin{proof}
    Suppose $(M,g)$ is simply connected and either the leaves of $\mathcal{D}$ are complete or $\xi$ is complete.
    If the leaves of $\mathcal{D}$ are complete with respect to $g$, then they are complete with respect to $\bar{g}=|V|^{-2} g $, because $|V|$ is constant along the leaves. 
    We apply~\ref{thm:splitting} to $\bar{g} = |V|^{-2} g$ and obtain that $(M,\bar{g})$ is isometric to a Riemannian product $(L,ds^2) \times (F,h)$, where $(L,ds^2)$ and $(F,h)$ are integral submanifolds of the foliations $\mathcal{L}$ and $\mathcal{D}$, respectively.
    Since the expressions obtained in~\ref{lem:step-1} depend only on the parameter along integral curves of $\xi$, the function $u:=|V|$ is constant along the leaves of $\mathcal{D}$.
    Therefore, under the product identification
    \[
        (M,\bar{g})\cong (L,ds^2)\times(F,h),
    \]
    there exists a positive smooth function, still denoted by $u$, defined on the one-dimensional factor $L$ such that
    \begin{equation}\label{eq:g-conformal-product}
        g=u(s)^2\bigl(ds^2+h\bigr).
    \end{equation}
    Since $L$ is one-dimensional and simply connected, we can identify it with an open interval $I\subset\mathbb R$.
    If we now introduce the new parameter $t$ by $dt=u(s)\,ds$, then $t\colon I\to J$ is a diffeomorphism onto an open interval $J\subset\mathbb R$, and~\eqref{eq:g-conformal-product} becomes
    \begin{equation}\label{eq:warped-global}
        (M,g)\cong \bigl(J\times F,\ dt^2+u(t)^2 h\bigr).
    \end{equation}
    Moreover, under this identification we have $\xi=\frac{\partial}{\partial t}$, $V=u(t)\,\partial_t$, and the leaves of $\mathcal D$ are exactly the slices $\{t\}\times F$.

    Analogously, if $\xi$ is complete, then by~\cite[Cor.~1]{PR1993} $(M,g)$ is isometric again to a warped product of the form~\eqref{eq:warped-global}, where $J=\mathbb R$ and $\xi=\frac{\partial}{\partial t}$.
    In this case, the leaves of $\mathcal D$ are again the slices $\{t\}\times F$.

    We now distinguish the two possible types.

    \noindent\textbf{Type~I.}
    Assume that $\psi$ is of type~I.
    Then $q_{\psi}=0$.
    By~\ref{prop:leaves}, each leaf of $\mathcal D$ endowed with the induced metric admits a non-trivial parallel $\Spin^c$-spinor.
    Since the induced metric on the slice $\{t\}\times F$ is $u(t)^2h$, and a constant rescaling preserves the existence of parallel $\Spin^c$-spinors, it follows that $(F,h)$ itself admits a non-trivial parallel $\Spin^c$-spinor.
    On the other hand, \ref{lem:step-1} yields constants $t_0$, $C\in \R$ such that $u(t)= C e^{2\int_{t_0}^{t}\mu(s)ds}$.
    Hence, absorbing $C^2$ into $h$, the metric is globally of the form
    \begin{equation}\label{eq:typeI-warped}
        g=dt^2+ e^{4\int_{t_0}^{t}\mu(s)ds}h,
    \end{equation}
    where $(F,h)$ carries a non-trivial parallel $\Spin^c$-spinor and $u$ satisfies the above first-order equation.

    \noindent\textbf{Type~II.}
    Assume now that $\psi$ is of type~II and $\dim M \geq 3$.
    Then $q_{\psi}>0$.
    By~\ref{lem:step-2}, $\mu$ is constant.
    Furthermore, for each $t\in J$, the slice $\{t\}\times F$ endowed with the induced metric $u(t)^2h$ admits two non-trivial real $\Spin^c$-Killing spinors with opposite Killing numbers $\pm \mu \frac{\sqrt{q_{\psi}}}{u(t)}$.
    Since, under the constant rescaling $u(t)^2h$ the Killing number is multiplied by $u(t)$, we conclude that $(F,h)$ admits two non-trivial real $\Spin^c$-Killing spinors with opposite Killing numbers $\pm \mu\sqrt{q_{\psi}}$.
    On the other hand, \ref{lem:step-1} yields
    \[
        u(t)=\sqrt{q_{\psi}}\,\sinh(2\mu t+C)
    \]
    on its maximal interval of positivity.
    Hence, after translating the parameter $t$ and changing its orientation if $\mu < 0$, we can assume that there exists $t_m \in (0,\infty]$ such that
    \[
        J=(0,t_m),\qquad u(t)=\sqrt{q_{\psi}}\,\sinh(2|\mu| t).
    \]
    Therefore, absorbing $q_{\psi}$ into $h$, the metric is globally of the form
    \begin{equation}\label{eq:typeII-warped}
        (M,g)\cong \bigl((0,t_m)\times F,\ dt^2+\sinh^2(2\mu t)\,h\bigr),
    \end{equation}
    where $(F,h)$ carries two non-trivial real $\Spin^c$-Killing spinors with opposite Killing numbers $\pm\mu$.
\end{proof}

Having obtained the simply connected models, it remains to understand how these models relate to the original manifold.
Let $\pi\colon(\widetilde M,\widetilde g)\to(M,g)$ be the universal Riemannian covering, and let $\Gamma$ be its deck transformation group.
We denote by a tilde the lift of the corresponding objects to $\widetilde M$. 
These lifted tensors, functions and vector fields are $\Gamma$-invariant.

\begin{lemma}\label{lem:step-4}
    Under the hypothesis~\eqref{assumption:H1}, assume that $V$ never vanishes and that either $\xi$ is complete or the leaves of $\mathcal D$ are complete.
    Applying~\ref{lem:step-3} to the lifted data on the universal cover $(\widetilde M,\widetilde g)$, we obtain a warped-product decomposition $\widetilde M \cong J \times \widetilde F$.
    Then, for every $\gamma\in\Gamma$, there exist a real number $a_\gamma$ and an isometry $\varphi_\gamma$ of $(\widetilde F,\widetilde h)$ such that $\gamma (t,x) = (t+a_\gamma,\varphi_\gamma(x))$ with respect to the above decomposition.
\end{lemma}
\begin{proof}
    The group $\Gamma$ acts by isometries on $(\widetilde M, \widetilde g)$ and on $\widetilde{\bar{g}} = |\widetilde{V}|^{-2} \widetilde g$.
    Since $\widetilde{\bar{g}}$ is a product metric and $\widetilde{\xi}$ is $\Gamma$-invariant, the action of $\Gamma$ preserves the line distribution generated by $\widetilde{\xi}$ and its orthogonal distribution. Hence it preserves the product structure.
    Therefore, every $\gamma\in\Gamma$ has one of the following two forms
    \[
        \gamma(t,x)=\bigl(t+a_\gamma,\varphi_\gamma(x)\bigr), \quad\text{or}\quad \gamma(t,x)=\bigl(-t+a_\gamma,\varphi_\gamma(x)\bigr),
    \]
    where $a_\gamma\in\mathbb R$ and $\varphi_\gamma$ is an isometry of $(\widetilde F,\widetilde h)$.
    Since $\widetilde{\xi}$ is invariant under the action of $\Gamma$, the second is impossible.
    Therefore, $\widetilde u(t+a_\gamma)=\widetilde u(t)$ for all $t\in J$ for which both sides are defined.
    If $\widetilde \xi$ is not complete, then $J \neq \mathbb{R}$. Since each deck transformation acts on $J$ by $t\mapsto t+a_\gamma$, it must satisfy $J + a_\gamma= J$, and therefore $a_\gamma=0$ for all $\gamma\in\Gamma$.
    In this case, $\Gamma$ acts trivially on the factor $J$ and by isometries on the factor $(\widetilde F,\widetilde h)$. 
\end{proof}

As a consequence of the previous discussion, we now obtain the final global descriptions in each of the two possible cases.

\begin{theorem}\label{thm:typeI-global}
    Let $(M^n,g)$ be a $\Spin^c$-Riemannian manifold admitting a generalized imaginary $\Spin^c$-Killing spinor $\psi$ of type~I, with nowhere-vanishing Dirac current $V$ and Killing function $i\mu$.
    Assume that the leaves of $\mathcal D=\ker(\xi^\flat)$ are complete or $\xi$ is complete.
    Then one of the following two possibilities holds: 
    \begin{itemize}
        \item $(M,g)$ is isometric to a warped product of the form
        \[
            (L\times F, dt^2 + e^{4\int_{t_0}^{t}\mu(s)\,ds}\,h),
        \]
        where $(L,dt^2)$ is a $1$-dimensional manifold, and $(F,h)$ admits a non-trivial parallel $\Spin^c$-spinor.
        \item the universal cover $(\widetilde M, \widetilde g)$ of $(M,g)$ is isometric to a warped product of the form
        \[
            (\R \times \widetilde F,\ dt^2+ e^{4\int_{t_0}^{t}\widetilde{\mu}(s)\,ds}\,\widetilde h),
        \]
        where $(\widetilde F,\widetilde h)$ admits a non-trivial parallel $\Spin^c$-spinor, and $\widetilde{\mu}$ is periodic with period $\tau > 0$ and satisfies $ \int_0^\tau \widetilde{\mu}(s)\,ds=0$.
        Furthermore, the deck transformation group of the universal cover acts by isometries of the form $\gamma(t,x)=\bigl(t+n_\gamma \tau,\varphi_\gamma(x)\bigr)$, with $n_\gamma\in\mathbb Z$ and $\varphi_\gamma$ an isometry of $(\widetilde F,\widetilde h)$.
    \end{itemize}
\end{theorem}
\begin{proof}
    In~\ref{lem:step-4}, we showed that every $\gamma\in\Gamma$ has the form
    \[
        \gamma(t,x)=\bigl(t+a_\gamma,\varphi_\gamma(x)\bigr),
    \]
    where $a_\gamma\in\mathbb R$ and $\varphi_\gamma$ is an isometry of $(\widetilde F,\widetilde h)$.
    Thus, if $a_\gamma = 0$ for every $\gamma\in\Gamma$, then $\Gamma$ acts trivially on the factor $J$ and by isometries on the factor $(\widetilde F,\widetilde h)$.
    Hence, passing to the quotient, we obtain that $(M,g)$ itself is a warped product of the form~\eqref{eq:typeI-warped}, with $F = \widetilde F / \Gamma$. This is the first possibility.
    By~\ref{lem:step-4}, if $\xi$ is not complete, or equivalently if $\widetilde \xi$ is not complete, then $a_\gamma=0$ for every $\gamma\in\Gamma$. Hence, in this case, we are always in the first possibility.

    Suppose now that there exists $\gamma\in\Gamma$ such that $a_\gamma\neq 0$.
    Since $\Gamma$ acts properly discontinuously, the set $A=\{a_\gamma:\gamma\in\Gamma\}$ is a discrete subgroup of $\mathbb R$.
    Hence, there exists $\tau>0$ such that $A=\tau \mathbb Z$.
    Without loss of generality, we denote by $\gamma$ the deck transformation such that $a_\gamma = \tau$.
    We also note that $\widetilde{u}(t)$ is $\tau$-periodic.
    Thus,
    \[
        \widetilde\mu(t)=\frac{\widetilde u'(t)}{2\widetilde u(t)}
    \]
    is also $\tau$-periodic.
    Moreover, by~\ref{lem:step-1}, there exists $C\in\R$ such that
    \[
        \widetilde u(t)=C\,e^{2\int_{t_0}^{t} \widetilde\mu(s)\,ds}.
    \]
    Since $\widetilde u$ is $\tau$-periodic, we obtain
    \[
        1=\frac{\widetilde u(t+\tau)}{\widetilde u(t)} = e^{2\int_t^{t+\tau}\widetilde\mu(s)\,ds}.
    \]
    Hence $\int_t^{t+\tau}\widetilde\mu(s)\,ds=0$ for all $t$.
    Since $\widetilde\mu$ is $\tau$-periodic, this is equivalent to $\int_0^\tau \widetilde\mu(s)\,ds=0$.
    This gives the second possibility.
\end{proof}

Before proving the type~II classification, we show a curvature consequence that is used to identify the auxiliary connection.

\begin{lemma}\label{lem:curvature-II}
    Under hypothesis~\eqref{assumption:H1}, if $\psi$ is of type~II, $V$ never vanishes and $\dim M \geq 3$, then the curvature $2$-form $\Omega^A$ of the auxiliary $\U(1)$-connection $A$ satisfies $V \lrcorner \Omega^A = 0$.
\end{lemma}
\begin{proof}
    Since, on a neighborhood, $\nabla^{g,A} = \nabla^g + \frac{i}{2} \omega$, where $\omega$ is the connection $1$-form of $A$, it follows that 
    \[
        R^{g,A} = R^g + \frac{i}{2} \Omega^A,
    \]
    where $R^g$ is the curvature of $\nabla^g$ and $R^{g,A}$ is the curvature of $\nabla^{g,A}$.
    By the Ricci identity~\cite[p.~16]{BFGK1991} (note that we have the opposite sign convention for the curvature), we have
    \[
        \frac{1}{2} \Ric^g(X) \cdot \psi = \sum_{i=1}^n e_i \cdot R^{g}_{Xe_i} \psi,
    \]
    where $\{e_i\}_{i=1}^n$ is a local orthonormal frame.
    Thus,
    \[
        \frac{1}{2} \Ric^g(X) \cdot \psi = \sum_{i=1}^n e_i \cdot R^{g,A}_{Xe_i}\psi - \frac{i}{2} \sum_{i=1}^n e_i \cdot \Omega^A(X,e_i) \psi.
    \]
    Since $\mu$ is constant when $\dim M \geq 3$,
    \begin{align*}
        R^{g,A}_{Xe_i}\psi &= \nabla^{g,A}_{[X,e_i]} \psi - \nabla^{g,A}_X \nabla^{g,A}_{e_i} \psi + \nabla^{g,A}_{e_i} \nabla^{g,A}_X \psi  = 2\mu^2 (e_i \cdot X + g(X,e_i)) \psi.
    \end{align*}
    Therefore, combining the previous two identities, we obtain
    \[
        \Ric^g(X) \cdot \psi = - 4\mu^2 (n-1) X \cdot \psi + i (X \lrcorner \Omega^A) \cdot \psi.
    \]
    We now take the inner product with $\psi$ and use that $\langle Y\cdot \psi,\psi\rangle$ is purely imaginary for every vector field $Y$ to obtain $i \langle (X \lrcorner \Omega^A)\cdot \psi,\psi\rangle = 0$.
    By the definition of Dirac current, we have $(X\lrcorner \Omega^A)(V) = (V \lrcorner \Omega^A)(X)$ for all $X \in \mathfrak{X}(M)$.
    This proves the lemma.
\end{proof}

\begin{theorem}\label{thm:typeII-global}
    Let $(M^n,g)$ be a $\Spin^c$-Riemannian manifold of dimension $n \geq 3$, admitting a generalized imaginary $\Spin^c$-Killing spinor $\psi$ of type~II, with nowhere-vanishing Dirac current $V$ and Killing function $i\mu$.
    Assume that the leaves of $\mathcal D=\ker(\xi^\flat)$ are complete or $\xi$ is complete.
    Then
    \begin{enumerate}
        \item $(M,g)$ is isometric to a warped product of the form~\eqref{eq:typeII-warped}.
        Let $(F,h)$ denote the Riemannian manifold appearing in this warped product.
        Then $(F,h)$ carries two non-trivial real $\Spin^c$-Killing spinors with opposite Killing numbers $\pm \mu$.
        \item The auxiliary $\U(1)$-connection is flat.
        \item If $(M,g)$ is not locally isometric to real hyperbolic space, then either $\dim M = 7$ and $(F,h)$ is strictly nearly Kähler, or $\dim M = 4k +2$ with $k\in\mathbb N$, $k \geq 1$, and $(F,h)$ is an Einstein--Sasakian manifold.
    \end{enumerate}
\end{theorem}

\begin{proof}
    By the global description obtained above in the type~II case, the vector field $\xi$ is not complete.
    Hence the assumption implies that the leaves of $\mathcal D$ are complete, and
    \[
        (M,g)\cong \bigl((0,t_m)\times F,\ dt^2+\sinh^2(2\mu t)\,h\bigr),
    \]
    where $(F,h)$ is complete and carries two non-trivial real $\Spin^c$-Killing spinors with opposite Killing numbers $\pm\mu$.
    This proves~(1).

    We next prove that the auxiliary curvature vanishes. 
    By Moroianu's classification~\cite[Cor.~4.2]{M1997} of simply connected $\Spin^c$-manifolds carrying real Killing spinors, the non-spin case occurs only for non-Einstein Sasaki manifolds endowed with the canonical or anticanonical $\Spin^c$-structure.
    Moreover, for each of these two fixed $\Spin^c$-structures, the corresponding real $\Spin^c$-Killing spinors occur with only one sign of the Killing number.
    Hence the existence, for the same auxiliary connection, of two real $\Spin^c$-Killing spinors with Killing numbers $\pm\mu$ excludes the non-spin case.
    Therefore the lifted $\Spin^c$-structure on $(\widetilde F,\widetilde h)$ must coincide with the spinorial Levi-Civita connection.
    In particular, the lifted auxiliary curvature vanishes, so $\widetilde\Omega^{A_F}=0$.
    Since $\widetilde\Omega^{A_F}$ is the pull-back of $\Omega^{A_F}$ by the universal covering map, it follows that $\Omega^{A}(\mathcal{D}, \mathcal{D}) =0$.
    By~\ref{lem:curvature-II}, we obtain $\Omega^A = 0$.
    Therefore, the auxiliary $\U(1)$-connection is flat.
    This proves~(2).

    To prove~(3), we apply Bär's classification~\cite{B1993} to the complete simply connected spin manifold $(\widetilde F,\widetilde h)$.
    Since $(\widetilde F,\widetilde h)$ carries two non-trivial real $\Spin^c$-Killing spinors with Killing numbers $\pm\mu$, exactly one of the following occurs: either $(\widetilde F,\widetilde h)$ is a round sphere, or $\dim \widetilde F=6$ and $(\widetilde F,\widetilde h)$ is strictly nearly Kähler, or $\dim \widetilde F=4k+1$ with $k\ge1$ and $(\widetilde F,\widetilde h)$ is an Einstein--Sasakian manifold.

    If $(\widetilde F,\widetilde h)$ is a round sphere, then $(F,h)$ is locally isometric to the round sphere.
    Hence the warped product metric
    \[
        dt^2+\sinh^2(2\mu t)\,h
    \]
    is locally the standard metric on real hyperbolic space, and therefore $(M,g)$ is locally isometric to $\mathbb H^n$.
    Thus, under the assumption that $(M,g)$ is not locally isometric to real hyperbolic space, this case is excluded.

    Assume now that $\dim F=6$.
    Let $\varphi$ be one of the real $\Spin^c$-Killing spinors on $(F,h)$, and let $\widetilde\varphi$ be its lift to $(\widetilde F,\widetilde h)$.
    By~(2), and since $(\widetilde F,\widetilde h)$ is simply connected, $\widetilde\varphi$ is a real Killing spinor invariant under the deck transformations.
    The strict nearly Kähler structure on $\widetilde F$ is determined algebraically
    by $\widetilde\varphi$, see~\cite{G1990} (or equivalently, the cone $C(\widetilde F)$ carries the parallel $G_2$-structure, whose tensor is induced by the corresponding parallel spinor).
    Since this tensor is defined from a deck-invariant spinor, it is itself deck-invariant and therefore descends to $(F,h)$.
    Consequently, $(F,h)$ is strictly nearly Kähler.

    The case $\dim F=4k+1$ is analogous.
    A lifted $\Spin^c$-Killing spinor on $(\widetilde F,\widetilde h)$ is deck-invariant and determines the Einstein--Sasakian structure on $\widetilde F$.
    Hence this structure descends to $(F,h)$, so $(F,h)$ is an Einstein--Sasakian manifold.

    Since $\dim M=\dim F+1$, we conclude that either $\dim M=7$ and $(F,h)$ is strictly nearly Kähler, or $\dim M=4k+2$ with $k\ge1$ and $(F,h)$ is an Einstein--Sasakian manifold.
    This proves~(3).
\end{proof}


\subsection{When the Dirac current vanishes}\label{sec:vanish}

The strategy in this subsection is local. 
We first analyze the behavior of the spinor near a zero of the Dirac current, then show that such zeros are isolated, and finally use the fact that $V$ is a closed conformal vector field to recover a radial warped-product description.

\begin{lemma}\label{lem:incomp}
    Under hypothesis~\eqref{assumption:H1}, if there exists $p\in M$ such that $V(p)=0$, then $\mu(p) \neq 0$.
\end{lemma}

\begin{proof}
    Assume by contradiction that $V(p)=0$ and $\mu(p)=0$.
    Since $\psi$ never vanishes, setting $\lambda:=2\mu|\psi|^2$, we also have $\lambda(p)=0$.
    By~\eqref{eq:conformal}, $\nabla^g_XV=\lambda X$, for all $X\in\mathfrak X(M)$. 
    Let $\gamma$ be any unit-speed geodesic with $\gamma(0)=p$, and define $s(t):=g(V(\gamma(t)),\dot\gamma(t))$.

    Then $\dot s(t)=g(\nabla^g_{\dot\gamma}V,\dot\gamma)=\lambda(\gamma(t))$.
    Moreover, the vector field $W(t):=V-s\dot\gamma$ satisfies
    \[
        \nabla^g _{\dot\gamma}W =\nabla^g_{\dot\gamma}V - \dot s\dot\gamma = \lambda\dot\gamma-\lambda\dot\gamma =0.
    \]
    Since $W(0)=V(p)=0$, it follows that $W$ vanishes along $\gamma$, hence
    $V(\gamma(t))=s(t)\dot\gamma(t)$.

    On the other hand, from $\nabla^g_XV=\lambda X$, one computes
    \[
        R^g_{XY}V=-X(\lambda)\,Y+Y(\lambda)\,X, \qquad \Ric^g(X,V)=-(n-1)X(\lambda).
    \]
        Taking $X=\dot\gamma$ and using $V=s\dot\gamma$, we obtain
    \[
        \dot \lambda=-\frac{1}{n-1}\Ric^g(\dot\gamma,\dot\gamma)\,s.
    \]
    Since $\dot s=\lambda$, this yields
    \[
        \ddot s+\frac{1}{n-1}\Ric^g(\dot\gamma,\dot\gamma)\,s=0.
    \]
    The initial conditions are $s(0)=0$, and $\dot s (0)=\lambda(0)=0$. 
    Hence, by uniqueness, $s=0$ along any geodesic starting from $p$.
    Thus both $V$ and $\lambda$ vanish on a normal neighborhood of $p$.
    Therefore, the set $\{x\in M:\ V(x)=0,\ \mu(x)=0\}$ is open and clearly closed as well.
    Since $M$ is connected and $\mu$ is not identically zero, this set must be empty.
    This contradiction proves that $\mu(p)\neq0$.
\end{proof}

We now show that the zeros of $V$ are isolated. 

\begin{lemma}\label{lem:isolated-zeros}
    Under hypothesis~\eqref{assumption:H1}, if there exists $p\in M$ such that $V(p)=0$, then $p$ is an isolated zero of $V$.
\end{lemma}

\begin{proof}
    By~\eqref{eq:conformal}, $\nabla^g _X V = 2 \mu |\psi|^2 X$,
    for all $X \in \mathfrak{X}(M)$.
    If $V(p) = 0$, then $(\nabla^g V)_p = 2 \mu |\psi(p)|^2 \Id_{T_p M}$ is a linear isomorphism because $\mu(p)\neq 0$ and $|\psi(p)| \neq 0$; thus, $V$ has an isolated zero at $p$.
\end{proof}

Finally, we prove the main result of this subsection.

\begin{theorem}\label{thm:vanishing-current}
    If a connected $\Spin^c$-Riemannian manifold $(M,g)$ with $\dim M \geq 3$ admits a generalized imaginary $\Spin^c$-Killing spinor $\psi$ whose Dirac current $V$ has zeros, then $(M,g)$ is locally isometric to the real hyperbolic space.
\end{theorem}

\begin{proof}
    Since generalized imaginary $\Spin^c$-Killing spinors never vanish (see~\cite[Lem.~3.1]{GN2015}) and the Dirac current $V$ vanishes somewhere, the spinor must be of type~II.
    Moreover, the open set where $V$ does not vanish is dense in $M$, and thus, by~\ref{lem:step-2} and continuity, the Killing function $\mu$ is constant everywhere.
    
    Let $r=d_g(p,\cdot)$.
    In the proof of~\ref{lem:incomp}, along every geodesic $\gamma$ starting at $p$, we proved that $V(\gamma(r)) = g(V(\gamma(r)),\dot\gamma(r))\dot\gamma(r)$; it follows that $\xi=\pm\partial_r$.
    Thus, $V$ and $\xi$ are radial vector fields.
    Moreover, the local reducibility of $\bar{g}=|V|^{-2}g$ implies that $g=dr^2+|V|^2h$, where $h$ is independent of $r$.
    Therefore, by~\ref{lem:step-1}, we have
    \[
        g=dr^2+q_{\psi}\,\sinh^2(2|\mu|r)\,h.
    \]
    We now use the formulas for the sectional curvatures of a warped product.
    Set $f(r)=\sqrt{q_{\psi}}\,\sinh(2|\mu|r)$.
    Let $X$, $Y$ be $h$-orthonormal vector fields tangent to the leaves of $\mathcal{D}$.
    Since $\dim M\ge 3$, such tangential $2$-planes exist.
    For the $g$-orthonormal tangential plane spanned by $\frac{1}{f}X$ and $\frac{1}{f}Y$, the sectional curvature is
    \[
        K_g\left(\frac{1}{f}X,\frac{1}{f}Y\right) =\frac{K_h(X,Y)-(\dot f)^2}{f^2}.
    \]
    On the other hand, since we are only interested in the behavior when $r\to 0$, we can write
    \[
        f(r)=2|\mu|\sqrt{q_\psi}\,r+O(r^3), \qquad \dot f(r)=2|\mu|\sqrt{q_\psi}+O(r^2).
    \]
    Hence
    \[
    K_g\left(\frac{1}{f}X,\frac{1}{f}Y\right) = \frac{K_h(X,Y)-4\mu^2q_\psi+O(r^2)}{4\mu^2q_\psi\,r^2+O(r^4)}.
    \]
    Since $g$ is smooth at $p$, its sectional curvature remains bounded as $r\to0$.
    It follows that the constant term in the numerator must vanish, and thus
    \[
        K_h(X,Y)=4\mu^2q_\psi
    \]
    for every tangential $2$-plane.
    Therefore $h$ has constant sectional curvature $4\mu^2q_\psi$.
    Consequently, after identifying the leaves locally with $S^{n-1}$, we may write
    \[
        h=\frac{1}{4\mu^2q_\psi}\,g_{S^{n-1}}.
    \]
    Substituting this into the expression for $g$, we obtain
    \[
        g=dr^2+q_\psi\sinh^2(2|\mu|r)\,\frac{1}{4\mu^2q_\psi}\,g_{S^{n-1}} =dr^2+\left(\frac{\sinh(2|\mu|r)}{2|\mu|}\right)^2 g_{S^{n-1}}.
    \]
    This is the polar form of the metric of the real hyperbolic space of constant sectional curvature $-4\mu^2$.
    Hence $(M,g)$ is locally isometric to the real hyperbolic space in a neighborhood of each point $p \in L = \{q\in M : \: V_q = 0 \}$.
    Since $L$ is discrete and $n\geq 3$, $M\smallsetminus L$ is connected.
    On the other hand, on $M\smallsetminus L$, the metric is locally isometric to the warped-product descriptions obtained above.
    In each such description, the equality $K_g=-4\mu^2$ is equivalent to the corresponding transverse metric having constant sectional curvature $4\mu^2q_\psi$.
    Since this holds on the charts meeting a neighborhood of $L$, the connectivity of $M\smallsetminus L$ implies that it holds on all of $M\smallsetminus L$.
    Thus $(M,g)$ is locally isometric to the real hyperbolic space.
\end{proof}

\subsection[When the distribution D is irreducible]{\texorpdfstring{When the distribution $\mathcal{D}$ is irreducible}{When the distribution D is irreducible}}\label{sec:irreducible}

In this subsection, we say that the distribution $\mathcal{D}=\ker(\xi^\flat)$ is \emph{irreducible} if its leaves are irreducible in the sense of de~Rham, that is, if the universal Riemannian covering of each leaf does not split as a non-trivial Riemannian product.
Under this assumption, we refine the nowhere-vanishing case.
The main point is that, in the type~I non-spin situation, the geometry of the leaves forces a transverse Kähler structure, which gives rise to an almost contact metric structure on $M$ of Kenmotsu type.

\begin{lemma}\label{lem:curvature-I}
    Let $(M,g)$ be a $\Spin^c$-Riemannian manifold admitting a generalized imaginary $\Spin^c$-Killing spinor $\psi$ of type~I with Dirac current $V$ and Killing function $i\mu$.
    Then the curvature $2$-form $\Omega^A$ of the auxiliary $\U(1)$-connection $A$ satisfies $V \lrcorner \Omega^A = 0$.
\end{lemma}

\begin{proof}
    The proof is analogous to that of~\ref{lem:curvature-II}, except that $\mu$ is not necessarily constant.
    We distinguish between the spin lift of the Levi-Civita connection $\nabla^g$ and the $\Spin^c$-connection $\nabla^{g,A}$ associated with $\nabla^g$ and the auxiliary connection $A$.
    We have, for all vector fields $Y$, $Z\in \mathfrak{X}(M)$,
    \[
        R^{g,A}_{YZ}\psi  = -i\,Y(\mu)\,Z\cdot\psi +i\,Z(\mu)\,Y\cdot\psi +\mu^2\bigl(Z\cdot Y\cdot-Y\cdot Z\cdot\bigr)\psi.
    \]
    Let $\{e_1,\dots,e_n\}$ be a local orthonormal frame.
    Since $X(\mu)=0$ for $X\in\mathcal D$, summing over $e_i$ yields
    \[
        \sum_{i=1}^n e_i\cdot R^{g,A}_{Xe_i}\psi = - 2\mu^2(n-1)\,X\cdot\psi +i\, \grad \mu \cdot X\cdot \psi.
    \]
    By the Ricci identity, we obtain
    \[
        \Ric^g(X)\cdot\psi = - 4\mu^2(n-1)\,X\cdot\psi +2i\, \grad\mu\cdot X\cdot\psi -i\,(X\lrcorner\Omega^A)\cdot\psi .
    \]
    Since $2\mu \xi^\flat$ is closed and $\xi^\flat$ is closed, $d\mu \wedge \xi^\flat = 0$.
    Thus, we can write $d\mu = g(\grad \mu, \xi) \, \xi^\flat$.
    Since $\psi$ is of type~I, it follows that $\grad \mu \cdot \psi = - i g(\grad \mu, \xi) \psi$.
    Indeed, for $X\in \mathcal{D}$,
    \[
        2i\,\grad\mu\cdot X\cdot\psi = - 2 g(\grad \mu, \xi) \, X\cdot \psi.
    \]
    Hence, for $X\in \mathcal{D}$, we take the inner product with $\psi$ and use that $\langle Y\cdot \psi,\psi\rangle$ is purely imaginary for every vector field $Y$ to obtain
    \[
        \Omega^A (X,V) = g(V, X \lrcorner \Omega^A) = 0.
    \]
    This proves the lemma.
\end{proof}

This lemma reduces the question of whether the $\Spin^c$-connection is induced by a spin structure to the corresponding question on the leaves of $\mathcal{D}$.
Indeed, this motivates the introduction of Kenmotsu geometry, which turns out to be the natural odd-dimensional counterpart of the transverse Kähler picture arising from type~I spinors.
An almost contact metric structure on a Riemannian manifold $(M,g)$ is a triple $(\mathcal{J},\xi,\eta)$ consisting of a $(1,1)$-tensor field $\mathcal{J}$, a unit vector field $\xi$ and a $1$-form $\eta$ satisfying 
\begin{equation*}
    \mathcal{J}^{2}X = -X + \eta(X)\,\xi, \qquad     g\bigl(\mathcal{J} X,\mathcal{J} Y\bigr) = g(X,Y) - \eta(X)\,\eta(Y),
\end{equation*}
for all vector fields $X$, $Y$ on $M$, see~\cite{B2010}.
Since we consistently use $\varphi$ for a spinor field, we denote the $(1,1)$-tensor field of an almost contact metric structure by $\mathcal{J}$ instead.

\begin{definition}\label{def:kenmotsu}
    Let $\alpha \in \mathcal{C}^{\infty} (M)$ be a function.
    An almost contact metric manifold $(M^{2m+1},g,\mathcal{J},\eta, \xi)$ is called an \emph{$\alpha$-Kenmotsu manifold} if, for any two vector fields $X$, $Y \in \mathfrak{X}(M)$, the following condition holds:
    \[
        (\nabla^g _X \mathcal{J} )Y = \alpha\big( g( \mathcal{J} X,Y)\,\xi -\eta(Y)\,\mathcal{J} X\big).
    \]
    This implies that $(\nabla ^g _X\eta)(Y) = \alpha\big(g(X,Y)-\eta(X)\, \eta(Y)\big)$, and hence $\nabla ^g _X \xi = \alpha\big(X-\eta(X)\,\xi\big)$, where $\nabla^g$ denotes the Levi-Civita connection of the metric~$g$.
    Furthermore, such a manifold is $\alpha$-Einstein if
    \[
        \Ric ^g(X, Y) = - \left((n-1)\alpha^2 + \xi(\alpha)\right) g(X, Y) - (n-2) \xi(\alpha) \eta(X) \eta(Y),
    \]
    for all $X$, $Y \in \mathfrak{X}(M)$.
\end{definition}
We refer to~\cite{K1972} for the properties of Kenmotsu manifolds.
These manifolds play the same role for imaginary $\Spin^c$-Killing spinors as Sasakian manifolds do for real $\Spin^c$-Killing spinors.
The $\alpha$-Einstein condition is a particular case of the $\eta$-Einstein manifolds; see~\cite{B2010} for a further exposition of this topic.

\begin{theorem}\label{thm:kenmotsu}
    Let $(M,g)$ be a simply connected $\Spin^c$-Riemannian manifold of dimension $n \geq 3$ and admitting a generalized imaginary $\Spin^c$-Killing spinor $\psi$ with nowhere-vanishing Dirac current $V$ and Killing function $i\mu$.
    Assume that the distribution $\mathcal{D} = \ker(\xi^\flat)$ is irreducible and its leaves are complete.
    Then one of the following two possibilities holds:
    \begin{itemize}
        \item The $\Spin^c$-connection over $(M,g)$ is induced by a spin structure.
        \item $(M,g)$ is a non-$(2\mu)$-Einstein Kenmotsu manifold.
    \end{itemize}
\end{theorem}

\begin{proof}
    If $\psi$ is of type~I or type~II, then by~\ref{lem:curvature-I} and~\ref{lem:curvature-II} we have $V \lrcorner \Omega^A = 0$, where $\Omega^A$ is the curvature $2$-form of the auxiliary $\U(1)$-connection $A$.
    Since $M$ is simply connected, the $\Spin^c$-connection $\nabla^{g,A}$ is induced by a spin structure if and only if $\Omega^A = 0$.
    Consequently, if $X \lrcorner \Omega ^A = 0$ for all $X \in \mathcal{D}$, then $(M,g)$ is spin.

    Assume first that $\psi$ is of type~II.
    By~\ref{prop:leaves}, the leaves of $\mathcal{D}$ carry two real $\Spin^c$-Killing spinors with opposite Killing numbers.
    Applying~\cite[Cor.~4.2]{M1997}, we conclude that the restricted $\Spin^c$-connection must be spin.

    Assume now that $\psi$ is of type~I.
    Again by~\ref{prop:leaves}, the leaves of $\mathcal{D}$ carry a non-trivial parallel $\Spin^c$-spinor.
    We now apply Moroianu's classification (see~\cite[Thm.~3.1]{M1997} and~\cite[Prop.~3.1]{M1997}) of simply-connected and irreducible $\Spin^c$-manifolds carrying parallel spinors to the universal Riemannian covering of each leaf of $\mathcal D$.
    There are two possibilities.
    Either the restricted auxiliary connection is flat, in which case the induced $\Spin^c$-connection on the leaves is induced by a spin structure, or the leaves are non-Ricci-flat Kähler manifolds and the restricted $\Spin^c$-structure is the canonical or the anticanonical one.
    The first possibility gives the first alternative of the theorem.
    We therefore assume from now on that the second possibility occurs.
    
    This determines a complex structure $J$ on the distribution $\mathcal{D}$.
    We now construct an almost contact metric structure on $(M,g)$.
    We set $\eta = \xi^\flat$ and define
    \[
        \mathcal{J}\colon TM \to TM \qquad \mathcal{J}(X) = J (X - g(X, \xi)\xi).
    \]
    A straightforward computation shows that $(M,g,\mathcal{J}, \eta, \xi)$ is an almost contact metric structure.
    We now check which type of almost contact metric structure it is.
    See~\cite{CG1990} for the classical classification and~\cite{ADDK2026} for a new perspective on this classification.
    We first note that the Levi-Civita connection of $g$ restricted to the leaves of $\mathcal{D}$ is given by
    \[
        \nabla^\top _X Y = \nabla^g_X Y+ 2\mu \, g(X,Y)\,\xi,
    \]
    for all $X$, $Y \in \mathcal{D}$.
    Let $\nabla$ also denote the restriction of the $(\alpha,\xi)$-connection to $\mathcal{D}$, where $\alpha = 2\mu$.
    Hence, again, $\nabla^\top _X Y = \nabla_X Y$ for all $X$, $Y \in \mathcal{D}$.
    Furthermore, by~\eqref{eq:spinor-paralelo}, we have, for $\phi = \frac{\psi}{|\psi|}$,
    \[
        \nabla_X \phi = 0 \qquad \forall X \in \mathfrak{X}(M).
    \]
    We denote by $\iota $ the inclusion map from the leaves of $\mathcal{D}$ to $M$.
    Thus, $\nabla^\top_X \iota^* \phi = 0$ for all $X\in \mathcal{D}$.
    By~\cite[Lem.~3.5]{M1997}, the leaves of $\mathcal{D}$ are Kähler manifolds with complex structure $J$ defined by
    \[
        JX\bullet (\iota^* \phi) = -iX \bullet (\iota^* \phi),
    \]
    where $\bullet$ denotes the Clifford multiplication of the restricted $\Spin^c$-structure on the leaves of $\mathcal{D}$.
    We now extend $J$ to a $(1,1)$-tensor field $\mathcal J$ on $TM$ by setting
    \[
        \mathcal J(\xi)=0, \qquad \mathcal J|_{\mathcal D}=J,
    \]
    or equivalently,
    \[
        \mathcal J Y \cdot \phi = -i\, (Y- g(Y,\xi)\xi) \cdot \phi, \qquad \forall Y \in \mathfrak{X}(M).
    \]
    Recall that these two descriptions are equivalent, because of our two conventions: $\xi \cdot \phi = - i \phi$ and $\bar X \bullet \iota^* \varphi = \iota^*(\bar X \cdot \xi \cdot \varphi)$, for all $\bar{X} \in \mathcal{D}$. 
    
    We now claim that $\mathcal{J}$ is $\nabla$-parallel.
    To prove this, we differentiate the previous identity with respect to $\nabla^A$, use $\nabla g= 0$ and $\nabla \xi = 0$, and obtain
    \[
        \nabla _X(\mathcal{J}Y) \cdot \phi = - i\, (\nabla _X Y - g(\nabla_XY,\xi ) \xi)\cdot \phi = \mathcal{J}(\nabla_X Y) \cdot \phi, 
    \]
    for all $X$, $Y \in \mathfrak{X}(M)$.
    Hence, since the Clifford multiplication is injective and $\phi$ never vanishes, we obtain $\nabla_X \mathcal{J}Y = \mathcal{J}(\nabla_X Y)$ for all $X$, $Y \in \mathfrak{X}(M)$, and thus $\nabla \mathcal{J} = 0$.

    We now compute $S_X \mathcal{J}$, where $\nabla = \nabla^g - S$.
    We have
    \begin{align*}
        (S_X \mathcal{J}) (Y) = S_X (\mathcal{J} Y) - \mathcal{J} (S_X Y) = - 2\mu g(X, \mathcal{J} Y) \xi - 2\mu g(Y,\xi)\mathcal{J} X.
    \end{align*}
    Since $\nabla^g \mathcal{J}= S \mathcal{J}$, we have proved that $(M,g, \mathcal{J}, \eta, \xi)$ is a $2\mu$-Kenmotsu manifold. 

    Finally, we determine when the leaves of $\mathcal D$ are Ricci-flat.
    Since $\nabla$ coincides with $\nabla^\top$ along the leaves, the restriction of the Ricci tensor of $\nabla$ to $\mathcal D$ agrees with the Ricci tensor of the leaves.
    Therefore, the leaves are Ricci-flat if and only if $\Ric(X,Y)=0$ for all $X$, $Y \in \mathcal{D}$.
    Moreover, because $\nabla \xi=0$, we also have $\Ric(X,\xi)=0$ for all $X\in\mathfrak X(M)$.
    On the other hand, for every $X\in\mathcal D$, \ref{P:Formulas-Curvature}, together with the fact that $X(2\mu)=0$, yields
    \[
        \Ric(\xi,X)=\Ric^g(\xi,X)=\Ric^g(X,\xi)=\Ric(X,\xi) - (n-2)X(\alpha)=0.
    \]
    Hence, all mixed Ricci components vanish, and the leaves of $\mathcal D$ are Ricci-flat precisely when $(M,g)$ is a $2\mu$-Einstein Kenmotsu manifold. 
    This completes the proof.
\end{proof}

\section[Generalized Imaginary Spin-c-Killing spinors of type~I]{\texorpdfstring{Generalized Imaginary $\Spin^c$-Killing spinors of type~I}{Generalized Imaginary Spin-c-Killing spinors of type~I}}\label{sec:type-I}

In this final section, we reverse the viewpoint developed in subsection~\ref{sec:charact}.
Instead of starting from a generalized imaginary $\Spin^c$-Killing spinor and extracting an $(\alpha,\xi)$-connection, we begin with a metric connection with vectorial torsion carrying a parallel $\Spin^c$-spinor and recover a generalized imaginary $\Spin^c$-Killing spinor of type~I.
This provides the promised connection-theoretic reformulation of the type~I equation.

\begin{theorem}\label{thm:Parallel to Imaginary}
    Let $(M,g)$ be a complete $\Spin^c$-Riemannian manifold of dimension $n\geq 3$, endowed with a non-flat $(\alpha,\xi)$-connection $\nabla$.
    Assume that $\xi \lrcorner \Omega^A = 0$.
    Then the following two statements are equivalent:
    \begin{enumerate}
        \item There exists a generalized imaginary $\Spin^c$-Killing spinor $\psi$ of type~I with normalized Dirac current $\xi$ and Killing function $i\mu$, where $\mu = \frac{\alpha}{2}$.
        \item There exists a nonzero $\nabla^A$-parallel $\Spin^c$-spinor $\phi$ whose Dirac current is $\xi$.
    \end{enumerate}
\end{theorem}

\begin{proof}
    The implication $(1)\Rightarrow(2)$ is a direct consequence of the analysis of~\ref{prop:principal-1}.
    Furthermore, by~\eqref{eq:spinor-paralelo}, $\phi = \frac{\psi}{|\psi|}$ is $\nabla^A$-parallel.
    
    We now prove that $(2)$ implies $(1)$.
    The proof has two steps.
    First, we show that $\alpha \xi^\flat$ is exact.
    Then we construct a generalized imaginary $\Spin^c$-Killing spinor of type~I from the $\nabla^A$-parallel $\Spin^c$-spinor $\phi$.
    
    Since $\phi$ is $\nabla^A$-parallel and $\xi$ is its Dirac current, we compute
    \begin{align*}
        (\nabla_X \xi^\flat) (Y) &= i X \left(\langle Y \cdot \phi, \phi \rangle \right) - i\langle \nabla_X Y \cdot \phi, \phi \rangle \\
        &= i \langle \nabla_X Y \cdot \phi, \phi \rangle -  i \langle \nabla_X Y \cdot  \phi, \phi \rangle = 0.
    \end{align*}
    Thus, $\nabla \xi = 0$ and, consequently, $\Ric(\cdot, \xi) = 0$.
    Furthermore, since $X \mapsto S_X \xi$ is a symmetric endomorphism, we have that $\xi^\flat$ is closed.

    We now use the Ricci identity for vectorial connections; see~\cite[Thm.~4.1]{AK2016}.
    Recall that we use the opposite sign convention for the curvature.
    We obtain
    \begin{equation*}
        \Ric (X) \cdot \phi = \sum_{i=1}^n e_i \cdot R_{Xe_i} \phi +  (d\alpha \wedge \xi^\flat \wedge X) \cdot \phi,
    \end{equation*}
    where $\Ric(X)$ denotes the Ricci endomorphism characterized by $g(\Ric(X),Y)=\Ric(X,Y)$.
    Since $R^A = R + \frac{i}{2} \Omega^A$ and $R^A_{XY}\phi = 0$ for all $X,Y\in\mathfrak{X}(M)$, we have that
    \begin{equation*}
        \Ric (X) \cdot \phi = -\frac{i}{2}  (X \lrcorner \Omega^A) \cdot \phi +  (d\alpha \wedge \xi^\flat \wedge X) \cdot \phi.
    \end{equation*}
    We now take $X = \xi$ and obtain that $\Ric(\xi) \cdot \phi = -\frac{i}{2} (\xi \lrcorner \Omega^A) \cdot \phi$.
    Since $\xi \lrcorner \Omega^A = 0$ and $\phi \neq 0$, it follows that $\Ric (\xi) = 0$.
    We now take $X \in \mathcal{D} = \ker(\xi^\flat)$ and use~\ref{P:Formulas-Curvature} together with $\Ric(X,\xi) = 0 = \Ric (\xi,X)$.
    We then obtain
    \begin{align*}
    0 = \Ric (\xi, X) = \Ric^g (\xi,X) = \Ric^g(X,\xi) = \Ric(X,\xi)- (n-2)X(\alpha) = - (n-2)X(\alpha).
    \end{align*}
    Thus, since $n\geq 3$, we have $X(\alpha) = 0$ for all $X \in \mathcal{D}$ and, consequently, $d\alpha \wedge \xi^\flat = 0$.
    In other words, $\alpha \xi^\flat$ is closed.
    
    We now consider the Weyl structure defined by $\nabla^W = \nabla - \alpha\xi^\flat\otimes \Id$.
    It is reducible in the direction of the Lee form $\alpha \xi^\flat$, so we are in the setting of~\cite[Thm.~1.1]{C2026}.
    By~\eqref{eq:Misma-Curvatura}, the connection $\nabla^W$ is not flat because $\nabla$ is not flat.
    Therefore,~\cite[Thm.~1.1]{C2026} implies that this Weyl structure is exact.
    Hence, $\alpha \xi^\flat$ is exact.

    Since $\alpha \xi^\flat$ is exact, there exists a function $f \in \mathcal{C}^{\infty} (M)$ such that $d f = \alpha \xi^\flat$.
    Since $\nabla \xi = 0$, the endomorphism $i\xi\cdot$ is $\nabla^A$-parallel.
    Moreover, $(i\xi\cdot)^2 = \Id$.
    Therefore, the spinor bundle splits as
    \[
        \Sigma_g^c M = \Sigma_g^+ M \oplus \Sigma_g^- M,
    \]
    where $\Sigma_g^{\pm} M$ are the eigensubbundles of $i\xi\cdot$ associated with the eigenvalues $\pm 1$.
    Let $\phi = \phi^+ + \phi^-$
    be the corresponding decomposition.
    Then $\nabla^A \phi^{\pm} = 0$ and $\xi\cdot \phi^{\pm} = \mp i\phi^{\pm}$.
    Since $\Sigma_g^+ M$ and $\Sigma_g^- M$ are orthogonal and the Dirac current of $\phi$ is $\xi$, we obtain
    \[
        1 = g(\xi,\xi) = i\langle \xi\cdot \phi,\phi\rangle = |\phi^+|^2 - |\phi^-|^2.
    \]
    In particular, $\phi^+$ is nowhere vanishing.

    We now define the spinor field $\psi = e^{f/2} \, \phi^+$.
    Since $\nabla^A \phi^+ = 0$, we have
    \[
        \nabla_X^A \psi = \nabla_X^A \left( e^{f/2}\, \phi^+\right) = \frac{X (f)}{2} \, e^{f/2} \, \phi^+ = \frac{\alpha }{2} \, g(X,\xi) \, \psi.
    \]
    Using that
    \[
        \frac{1}{2}\, S_X \cdot \psi =  - \frac{\alpha}{2} \,  X \cdot \xi \cdot \psi -\frac{\alpha}{2} \, g(X,\xi)\, \psi,
    \]
    we obtain
    \[
        \nabla_X^{g,A} \psi = \nabla_X^A \psi + \frac{1}{2} S_X \cdot \psi = -\frac{\alpha}{2} \,  X \cdot \xi \cdot \psi = \frac{\alpha}{2} i X\cdot \psi.
    \]
    Therefore, $\psi$ is a type~I generalized imaginary $\Spin^c$-Killing spinor (since $\xi \cdot \psi = - i \psi$) with Killing function $i\mu$, where $\mu = \frac{\alpha}{2}$.
    This completes the proof.
\end{proof}

\begin{remark}
    The correspondence in~\ref{thm:Parallel to Imaginary} is explicit: a type~I generalized imaginary $\Spin^c$-Killing spinor $\psi$ is sent to
    \[
        \phi=\frac{\psi}{|\psi|},
    \]
    which is $\nabla^A$-parallel. Conversely, a nonzero $\nabla^A$-parallel spinor $\phi$ whose Dirac current is $\xi$ gives
    \[
        \psi=e^{f/2}\phi^+,
    \]
    where $df=2\mu\xi^\flat$ and $\phi^+$ is the component satisfying $\xi\cdot\phi^+=-i\phi^+$.
\end{remark}

\printbibliography
\end{document}